\documentclass[12pt,a4paper]{amsart}
	
	\usepackage{amssymb,amsmath,amsthm,amsfonts,amscd,latexsym, dsfont, color}
        \usepackage[dvipsnames]{xcolor}
	\usepackage[hidelinks]{hyperref} 
	\usepackage{graphics}
	\usepackage{wrapfig}
	\usepackage{pst-node}
	\usepackage{tikz-cd} 
	\usepackage{enumitem}
	\usepackage{mathrsfs}
	\newcommand{\bbR}{\mathbb R}

	\newcommand{\bbC}{\mathbb C}
	\newcommand{\bbH}{\mathbb H}
    \newcommand{\CP}{\mathbb{CP}}
    \newcommand{\bbD}{\mathbb D}

	\usepackage{cleveref}

    \newcommand{\calC}{\mathcal{C}}

\DeclareFontFamily{U}{mathx}{\hyphenchar\font45}
\DeclareFontShape{U}{mathx}{m}{n}{<-> mathx10}{}
\DeclareSymbolFont{mathx}{U}{mathx}{m}{n}
\DeclareMathAccent{\widebar}{0}{mathx}{"73}

\crefformat{section}{\S#2#1#3}
\crefformat{subsection}{\S#2#1#3}
\crefformat{subsubsection}{\S#2#1#3}
	\allowdisplaybreaks

\theoremstyle{definition}

\usepackage{thmtools, thm-restate}
\theoremstyle{plain}

\newtheorem{theorem}{Theorem}[section]
\newtheorem{proposition}[theorem]{Proposition}
\newtheorem{lemma}[theorem]{Lemma}
\newtheorem{corollary}[theorem]{Corollary}
\newtheorem*{question}{Question}

\newtheorem{example}[theorem]{Example}

\newtheorem{definition}[theorem]{Definition}

\newtheorem{remark}{Remark}

\numberwithin{equation}{section}

	\usepackage[margin=1in]{geometry}

\newcommand{\qcg}{\mathcal{QC}^{\dagger}}
\newcommand{\eqcg}{\mathcal{E\!QC}^{\dagger}}
\newcommand{\Aut}{\mathrm{Aut}}
\newcommand{\qag}{\mathcal{QA}^{\dagger}}
\newcommand{\nqag}{\mathcal{N\!QA}^{\dagger}}
\newcommand{\lqag}{\mathcal{QA}_{Lk}^{\dagger}}

\newcounter{foo}

\newtheorem{theo}[foo]{Theorem}
\newtheorem{cor}[foo]{Corollary}

\begin{document}

\title{Graphs of Quasicircles and Quasiconformal Homeomorphisms}
\author{Katherine Williams Booth, Alexander Nolte, and Yvon Verberne}
\date{}

\begin{abstract}
    We give a combinatorial characterization of the group of quasiconformal homeomorphisms of a closed, oriented surface $S$ of genus at least $2$.
    In particular, we prove they are exactly the automorphisms of a graph of essential quasicircles on $S$ that respect a canonical coarse ordering induced by quality constants. We also discuss the coarse geometry of this graph.
\end{abstract}
\maketitle

\stepcounter{section}

Bowden--Hensel--Webb recently proved \cite{bowden2022quasimorphisms} that for closed, oriented surfaces $S$ of genus $g\geq 1$, the identity component $\text{Diff}_0(S)$ of the diffeomorphism group of $S$, while simple \cite{thurston1974foliations}, is not uniformly perfect. This means that there is no uniform $N$ so that any diffeomorphism $h \in \text{Diff}_0(S)$ may be written as a product of $N$ commutators. This came as a surprise, as the analogous statement is false for spheres and for closed manifolds of any dimension other than $2$ or $4$ \cite{burago2008conjugation,tsuboi2008uniform,tsuboi2012uniform}.

The proof of Bowden--Hensel--Webb introduced and centered on an analysis of the \textit{fine curve graph} $\mathcal{C}^{\dagger}(S)$ of $S$, whose vertices are topologically embedded essential curves in $S$ with edges specified by disjointness. The fine curve graph differs from more-studied curve graphs (e.g. \cite{behrstock2006asymptotic,bestvinaBromberg2019asymptotic,bestvinaFujiwara2002bounded,harer1985stability,masur1999geometry,masurMinsky2000geometryII}) in that curves in $\mathcal{C}^\dagger(S)$ are \textit{not} considered up to isotopy, hence all vertices of $\mathcal{C}^\dagger(S)$ have uncountable valence.

This has led to interest in fine curve graphs, e.g. \cite{boothMinahanShapiro2023automorphisms,bowdenHenselMannMiltonWebb2022rotation,guiheneufMiliton2024parabolic,leRouxWolff2024automorphisms,long2023automorphisms}.
The guiding motivation here is that combinatorial tools from geometric group theory may be able to see aspects of homeomorphism groups of surfaces that are otherwise obscure.

As a promising step in this program, Farb--Margalit \cite{farbMargalit2022Unpublished} and Long--Margalit--Pham--Verberne--Yao \cite{long2023automorphisms} showed that the automorphism group of the original fine curve graph is canonically isomorphic to the group $\text{Homeo}(S)$ of homeomorphisms of $S$.
It is much less clear to what extent classical subgroups of $\mathrm{Homeo}(S)$ are accessible from this perspective.

In particular, the automorphism groups of \textit{subgraphs} of $\mathcal{C}^{\dagger}(S)$ consisting of curves with regularity properties are mysterious and seem less well-behaved.
For instance, Le~Roux--Wolff \cite{leRouxWolff2024automorphisms} proved that not all automorphisms of the subgraph of smooth curves are induced by diffeomorphisms, and the first named author proved the automorphism group of the $C^1$-curve graph is substantially larger than the group of $C^1$ diffeomorphisms \cite{booth2024homeomorphisms}.

\vspace{0.3cm}

The point of this paper is to characterize the group $\mathcal{Q}(S)$ of quasiconformal homeomorphisms of a compact oriented surface $S$ in this circle of ideas.
Quasiconformality is a much-studied weakening of conformality with a range of applications, for instance in analyzing deformations of complex structures in Teichm\"uller theory.

We note that we allow quasiconformal homeomorphisms to reverse orientation, i.e.~$\varphi : S \to S$ is quasiconformal if it is quasiconformal with respect to a choice of orientations on the base and target.

 This is the first classically-studied group beyond $\mathrm{Homeo}(S)$ to be characterized in terms of automorphisms of a fine curve graph. The relevant graph $\mathcal{QC}^{\dagger}(S)$ has vertices essential quasicircles in $S$ and edges given by disjointness.
See \S \ref{s-qcbasics} and \S \ref{s-main-defs} for definitions.
The theory of quasicircles and their connections to quasiconformal homeomorphisms goes back to work of Ahlfors \cite{ahlfors1963quasiconformal}, and is by now a well-developed and broad area (e.g. \cite{bishop2025weilPetersson,bonsanteDancigerMaloniSchlenker2021quasicircles,gehring1999characterizations,smirnov2010dimension,tukiaVaisala1980quasisymmetric}).

We do not show that every automorphism of $\mathcal{QC}^{\dagger}(S)$ is induced by a quasiconformal homeomorphism.
At least in the case of surfaces with punctures, there are counter-examples to this expectation, see Example \ref{ex-counter-ex}.

However, $\qcg(S)$ has extra structure that comes from quality constants $K \geq 1$ that measure the irregularity of quasicircles with respect to an auxiliary hyperbolic metric.
We say that $\Phi \in \mathrm{Aut}(\qcg(S))$ is \textit{coarsely order preserving} if for all $K\geq 1$ there is some $K' < \infty$ so that the image of any $K$-quasicircle under $\Phi$ is at most a $K’$-quasicircle. See \S \ref{s-qcbasics} and \S \ref{s-main-defs} for precise definitions and more discussion.
We indicate that choosing definitions here is a point of some delicacy because not all characterizations of quasicircles have uniformly comparable involved quality constants.

Our main theorem is the following:

\begin{theo}\label{thm-first}
    Let $S$ be a closed orientable surface of genus $g \geq 2$.
    Any automorphism $\Phi$ of $\mathcal{QC}^\dagger(S)$ is induced by a homeomorphism $\varphi$ of $S$.
    The homeomorphism $\varphi$ is quasiconformal if and only if $\Phi$ is coarsely order preserving.
\end{theo}

In particular, the canonical map $\mathcal{Q}(S) \to {\rm{Aut}}\mathcal{QC}^{\dagger}(S)$ is an isomorphism onto the subgroup ${\rm{CAut}}\mathcal{QC}^{\dagger}(S)$ of coarsely order preserving automorphisms.

Theorem \ref{thm-first} gives a combinatorial characterization of the group $\mathcal{Q}(S)$.
From the perspective of fine curve graphs, the main novelties of Theorem \ref{thm-first} are the appearance of a new subgroup of interest and the notion of coarse ordering; from the analytic perspective, the main novelty is the \textit{combinatorial} flavor of the characterization of quasiconformal maps.
Much of the point of Theorem \ref{thm-first} is to make a connection between the two areas, and the result is proved by combining tools from the geometric-group-theoretic study of fine curve graphs with analytic tools on quasicircles and quasiconformal mappings.

This result opens the classical group $\mathcal{Q}(S)$ to study with techniques from geometric group theory.
As a step towards the usefulness of Theorem \ref{thm-first} in this direction, we prove:

\begin{theo}\label{thm-hyperbolic}
    $\qcg(S)$ is Gromov hyperbolic for closed oriented surfaces $S$ of genus $g \geq 2$.
\end{theo}

Knowledge of the hyperbolicity of $\qcg(S)$ is desirable, because the tools of geometric group theory are particularly well-adapted to the study of groups through actions on Gromov-hyperbolic spaces. Bowden--Hensel--Webb \cite{bowden2022quasimorphisms} leverage the hyperbolicity of $\mathcal{C}^{\dagger}(S)$ to show that the space $\widetilde{QH}(\text{Diff}_0(S))$ of unbounded quasimorphisms on $\text{Diff}_0(S)$ is infinite dimensional, proving that $\text{Diff}_0(S)$ is not uniformly perfect.

As a consequence of the differential characterization of quasiconformality and compactness of $S$, any $C^1$ diffeomorphism is a $K$-quasiconformal map for some $K$. So using Theorem~\ref{thm-hyperbolic}, we can follow the arguments in \cite[Sections 4 \& 5]{bowden2022quasimorphisms} to obtain:

\begin{cor}
    For a closed oriented surface $S$ with genus $g\geq 2$, the space $\widetilde{QH}(\mathcal{Q}(S))$ of unbounded quasimorphisms on $\mathcal{Q}(S)$ is infinite dimensional.
\end{cor}

We conclude by remarking that we do not know whether or not the assumption of coarse order preservation is necessary in the closed case.
This seems to require understanding subtle aspects of the structure of general quasiconformal maps, and so we ask:

\begin{question}\label{q-is-filtration-needed}
    Is every homeomorphism of a closed surface that sends quasicircles to quasicircles quasiconformal? 
\end{question}

\subsection{Outline}
We begin by recalling fundamentals on quasicircles and quasiconformal homeomorphisms in \S \ref{s-qcbasics}.
We then introduce $\mathcal{QC}^\dagger(S)$ and its basic structure in \S \ref{s-main-defs}.

Proving that automorphisms of various fine curve graphs are induced by homeomorphisms is in general an involved process in which the particulars depend heavily on the graph under consideration.
There is not a known general result that includes $\mathcal{QC}^\dagger(S)$, but we are able to show that any automorphism of $\mathcal{QC}^\dagger(S)$ is induced by a homeomorphism with a combination of methods from the continuous and smooth cases \cite{booth2024automorphisms,long2023automorphisms}.
We explain how to adapt these methods to our setting in \S \ref{Section:TopologyToGraphs}-\ref{ss-get-homeo}.
We then prove that a homeomorphism that induces an element of $\mathrm{CAut}(\qcg(S))$ is quasiconformal in \S \ref{ss-quasiconformality}.
Finally, we prove Theorem~\ref{thm-hyperbolic} in \S 6. 

\subsection*{Acknowledgements}
This project began at the Tech Topology Conference in December 2022. The authors would like to thank Dan Margalit for helpful comments on an earlier draft.
K.W.B. was supported by the National Science Foundation under DMS Grants No. 2417920, and 2503473.
A.N. was supported by the National Science Foundation under DMS Grants No. 1842494, 2502952, and 2005551.
Y.V. was partially supported by an NSERC-PDF Fellowship and by an NSERC Discovery grant RGPIN 05587.

\section{Quasiconformal Homeomorphisms and Quasicircles}\label{s-qcbasics} 

Let us set some definitions for quasicircles in surfaces. 
There are many characterizations of quasicircles (e.g. \cite{gehring1999characterizations}), some of which give non-comparable quality constants.
What is important for our proofs here is that our characterization is based on M\"obius, not metric, geometry.
Standard references for the basic material here are \cite{lehto1973quasiconformal} and \cite{lehto2012univalent}.

Recall also that a homeomorphism $\varphi$ of a Riemann surface $\Sigma$ is \textit{$K$-quasiconformal} if it distorts the moduli of ring domains in $\Sigma$ by a factor of at most $K$.
Quasiconformality admits a reformulation (e.g. \cite{bers1960quasiconformal}) as a requirement of $\varphi$ to have $L^2$ generalized partial derivatives that almost everywhere satisfy the differential relation corresponding to distortion of eccentricity of infinitesimal ellipses by at most a factor of $K$.
In particular, $K$-quasiconformality is a local condition and any $C^1$ diffeomorphism of a closed surface is $K$-quasiconformal for some $K$.

Let $\mathcal{Q}(S)$ denote the group of quasiconformal homeomorphisms of a Riemann surface structure $\Sigma$ on a closed topological surface $S$.
As the notation suggests, $\mathcal{Q}(S)$ does not depend on the choice of Riemann surface structure due to the closedness of $S$. 

\subsection{Quasicircles in the Riemann sphere} In contrast to $K$-quasiconformality, $K$-quasicircularity is \textit{not} a local condition. 
This leads us to take some care with its treatment on surfaces with genus.
On $\CP^1$, we recall:

\begin{definition}
    A $K$-quasicircle in $\CP^1$ is the image of a circle under a $K$-quasiconformal homeomorphism of $\CP^1$.
\end{definition}

There are many equivalent reformulations of the condition of quasicircularity on $\CP^1$ and notions of quality constants defining $K$-quasicircularity.
We recall a reformulation that will be useful to us.
A curve $\gamma \in \bbC$ satisfies the \textit{$4$-point $c$-bounded turning condition} for $c > 0$ if for all distinct and circularly ordered $p_1, p_2, p_3, p_4 \in \gamma$, the cross ratio $(p_4, p_1; p_2,p_3)$ is no less than $1/c$, i.e. \begin{align*}
    \frac{| p_1 - p_2| | p_4 - p_3 |}{|p_1 - p_3| |p_4 - p_2|} \leq c.
\end{align*}
There is a standard modification to extend to the Riemann sphere $\CP^1$.
Because the cross-ratio is M\"obius-invariant, the $4$-point $c$-bounded turning condition is invariant under M\"obius transformations of $\CP^1$.

The $4$-point bounded turning condition characterizes quasicircularity, with uniform dependence in the involved constants:

\begin{theorem}[Ahlfors {\cite[Theorem 1]{ahlfors1963quasiconformal}}]
    For every $K > 0$ there is a $c(K) < \infty$ so that every $K$-quasicircle in $\CP^1$ satisfies the $4$-point $c$-bounded turning condition.
    Conversely, for every $c > 0$ there is a $K(c) < \infty$ so that every Jordan curve in $\CP^1$ that satisfies the $4$-point $c$-bounded turning condition is a $K(c)$-quasicircle.
\end{theorem}

In the opposite direction of the definition of quasicircularity in terms of quasiconformal maps, quasiconformal maps of $\CP^1$ are characterized in terms of their distortion of circles.

\begin{theorem}[Aseev \cite{aseev2014quasiconformal}]\label{thm-aseev}
    Let $\Omega \subset \bbC$ be a domain, and let $f: \Omega \to \bbC$ be an injection.
    That $f$ is a quasiconformal homeomorphism is equivalent to the existence of a $K \geq 1$ so that for every $z \in \Omega$ there is a neighborhood $U_{z}$ so that for every circle $C$ in $U_z$ the image $f(C)$ is a $K$-quasicircle.
\end{theorem}

We note that this theorem considers quasiconformality without the requirement of orientation-preservation.

\subsection*{Extension results} In producing quasiconformal maps, we shall find extension results in $\CP^1$ useful.
In particular, we shall use the following two theorems.
First, we shall use the smoothness in the following result of Tukia.

\begin{theorem}[Tukia {\cite{tukia1981extension}}]\label{thm-smooth-away-from-bad}
    Let $\gamma \subset \CP^1$ be a quasicircle.
    Then there is a quasiconformal map of $\CP^1$ taking $\bbR$ to $\gamma$ that is smooth away from $\gamma$.
\end{theorem}
This exact statement follows from the main result of \cite{tukia1981extension} and \cite[Theorem 4.9]{tukiaVaisala1980quasisymmetric}, as discussed on page 2 of \cite{tukia1981extension}.
We note that the smoothness we use comes from properties of the Beurling--Ahlfors extension \cite{beurlingAhlfors1956boundary}.

Second, we shall use the following extension result.
\begin{theorem}\label{thm-bilip-extension}
    Let $L \geq 1$. Then there is a $K = K(L)$ so that every $L$-bilipschitz embedding $f: S^1 \to \bbD$ with respect to the Euclidean metric on $\bbD$ admits a $K$-quasiconformal extension to the disk.
\end{theorem}

This is a corollary of the planar Lipschitz Sch\"onflies theorem \cite[Theorem A]{tukia1980planarScoenflies}, together with the observation (e.g. \cite[\S 7]{tukia1981extension}) that for every $L > 1$ there is a $K(L)$ so that every $L$-bilipschitz embedding $\bbD \to \bbD$ is $K$-quasiconformal.

\subsection{Surfaces with genus} 
In higher genus, some care is needed in choosing a characterization of quasicircles in which the quality constants behave well for essential curves.
We use the following definition:

\begin{definition}\label{def-qcirc-on-genus}
    Let $S$ be a closed, oriented surface with a metric $g$ of constant sectional curvature $0$ or $-1$. An {\rm{essential $K$-quasicircle}} in $(S, g)$ is the image of a simple closed geodesic $\gamma$ in $S$ under a $K$-quasiconformal homeomorphism of $S$.
\end{definition}

Here, $K$-quasiconformality is with respect to the induced conformal structure of $(S,g)$.
This characterization is adapted to M\"obius geometry as these metric geodesics are circular arcs in the induced $\mathbb{CP}^1$-structures. 
Definition \ref{def-qcirc-on-genus} gives the usual family of quasicircles (Remark \ref{remark-def-well-behaved}).
We remark that the quasicircularity constants in our definition have arbitrarily large differences with those appearing in metric definitions of quasicircularity.
A source of this discrepancy is that simple geodesics may wrap back near themselves as they traverse the topology of $S$.

On a practical level, Definition \ref{def-qcirc-on-genus} is convenient for working with curves known to be quasicircles, but requires the production of a quasiconformal mapping in the verification that a curve is a quasicircle.

It will help us, at times when we are not considering quality constants, to be able to swap between characterizations of quasiconformality.
To that end, we recall that an embedded curve is a quasicircle if and only if it satisfies the (two-point) metric bounded turning condition: there is a $k< \infty$ so that for all $x,y \in \gamma$ there is a connected component of $\gamma - \{x,y\} $ of diameter bounded by $kd(x,y)$.

\begin{remark}\label{remark-def-well-behaved}
    One may verify that Definition \ref{def-qcirc-on-genus} is equivalent to the metric bounded turning characterization, with potentially non-comparable quasicircularity constants.
    
    Indeed, as simple closed geodesics are metric bounded turning quasicircles, the implication that any image of a simple closed geodesic under a quasiconformal map is a metric quasicircle follows.
    The other direction is obtained as follows.
    Let $\gamma$ be a metric bounded turning quasicircle in $S$ and $\gamma_0$ be the geodesic in the free homotopy class of $\gamma_0$.
    Conformally identify an annular neighborhood of $\gamma$ with a ring domain in $\mathbb{C}$ and a metric $\varepsilon$-neighborhood of $\gamma_0$ with an annulus in $\mathbb{C}$ in which $\gamma_0$ is the circle.
    By Theorem \ref{thm-smooth-away-from-bad}, there is a quasiconformal homeomorphism of a neighborhood of $\gamma_0$ to a neighborhood of $\gamma$ that takes $\gamma_0$ to $\gamma$ and is smooth away from $\gamma_0$.
    From smoothness, extending quasiconformally to the rest of $S$ is classical.
\end{remark}

Parts of our arguments below that do not consider quality constants make use of inessential curves.
We never refer to quality constants for inessential quasicircles.
To be concrete, we adopt the convention that inessential curves on $S$ are quasicircles if they satisfy the two-point metric bounded turning condition.

Definition \ref{def-qcirc-on-genus} is independent of the choice of metric in the following sense. 
\begin{lemma}\label{lemma-qc-independence}
    Suppose $g$ and $g'$ are locally homogeneous metrics on $S$. Then for every isotopy class $[\alpha]$ of essential simple closed curves in $S$, there is a $K(\alpha)$ so that if $\gamma \in [\alpha]$ is a $K'$-quasicircle with respect to $g$ then $\gamma$ is a $K(\alpha)K'$-quasicircle with respect to $g'$.
\end{lemma}

One may see Lemma \ref{lemma-qc-independence} using the argument in Remark \ref{remark-def-well-behaved}, though the smoothness of geodesics makes Theorem \ref{thm-smooth-away-from-bad} a far deeper result than is necessary.

We conclude by mentioning that the metric bounded turning characterization of quasicircles localizes in the following manner.
The involved quality constants are not uniform.

\begin{lemma}[Non-uniform localization]\label{lemma-bt-localizes}
    Let $\gamma \subset S$ be an embedded circle that admits an open covering by neighborhoods $U_i$ on which it satisfies the metric bounded turning condition for constants $k_i$.
    Then $\gamma$ is a quasicircle.
\end{lemma}

\begin{proof}
    Take a finite sub-cover $U_1, ..., U_n$ of $\gamma$.
    The existence of a Lebesgue number $\delta$ for the subcover implies that the metric bounded turning condition holds for some $k$ for all $\delta$-close pairs of points on $\gamma$, and compactness implies that for points separated by at least $\delta$ satisfy a $k'$-metric bounded turning inequality.
\end{proof}

\section{Graphs of Quasicircles}\label{s-main-defs}

The combinatorial objects we consider are analogues of the \textit{fine curve graphs} $\mathcal{C}^\dagger(S)$ of \cite{bowden2022quasimorphisms}, whose vertices are simple closed curves in $S$ with edges given by disjointness.
The following are the central objects of our study:

\begin{definition}[Quasicircle graph]
    For $S = S_g$ with $g \geq 2$, let $\mathcal{QC}^\dagger(S)$ be the graph with vertices the set of essential quasicircles in $S$ and edges between vertices representing disjoint quasicircles.
\end{definition}

Lemma \ref{lemma-qc-independence} shows that $\mathcal{QC}^\dagger(S)$ is a well-defined graph.
In fact, more information is invariant under change of reference data. 
To state this invariance, let us adopt the following notation.

\begin{definition}
    Let $S = S_g$ with $g \geq 2$.
    Let $\mathscr{S}$ denote the collection of isotopy classes of essential simple closed curves in $S$ and let $[\beta] \in \mathscr{S}$.
    Then define $\mathcal{QC}^\dagger_{[\beta]}(S)$ as the subgraph of $\mathcal{QC}^\dagger(S)$ induced by the vertices in the isotopy class $[\beta]$.
\end{definition}

We show in \S\ref{Section:BuildQCFromGraph}  that all automorphisms $\Phi$ of $\mathcal{QC}^\dagger(S)$ are induced by homeomorphisms $\varphi$ of $S$. 
We adopt the notation $\Phi([\beta]) = \varphi([\beta])$.
The (unlabeled) partition $\qcg(S) = \bigsqcup_{[\beta]\in \mathscr{S}} \qcg_{[\beta]}(S)$ is invariant under automorphisms of $\qcg(S)$ and so is well-defined.

There is a filtration $\mathcal{F} = \{ \mathcal{F}_K\}_{K \in [1,\infty)}$ of $\qcg(S)$ with $\mathcal{F}_K$ the set of $K$-quasicircles on $S$.
By \textit{filtration}, we mean that $\mathcal{F}_{K} \subset \mathcal{F}_{K'}$ for any $K < K'$ and $\mathcal{QC}^\dagger(S) = \bigcup_{K \geq 1}\mathcal{F}_K$.
The filtration $\mathcal{F}$ restricts to filtrations $\mathcal{F}_{[\beta]} = \{ \mathcal{F}_{K, [\beta]}\}_{K \in [1,\infty)}$ of each subgraph $\qcg_{[\beta]}(S)$.
Of course, $\mathcal{F}$ is not invariant under change of hyperbolic metric or application of a quasiconformal map of $S$.
But changing metrics respects $\mathcal{F}$ in the following weakly coarse sense (Lemma \ref{lemma-qc-independence}):

\begin{definition}\label{def-coarse-pres}
An automorphism $\Phi \in {\rm{Aut}}\qcg(S)$ is {\rm{weakly coarse order preserving}} if for every $K \geq 1$ and $[\beta] \in \mathscr{S}$ there is a $K' = K'(K,\beta) > 1$ so that $\Phi\left(\mathcal{F}_{K, [\beta]}\right) \subset \mathcal{F}_{K', \Phi([\beta])}$.
We say $\Phi$ is {\rm{coarsely order preserving}} if the constants $K'(K,\beta)$ can be taken to be independent of $\beta$.
    Let ${\rm{CAut}}\qcg(S)$ denote the group of weakly coarse order preserving automorphisms of $\qcg(S)$.
\end{definition}

Note that any $K$-quasiconformal homeomorphism $\varphi$ of $S$ induces a graph automorphism, which we denote by $\varphi_*$, of $\mathcal{QC}^\dagger(S)$. 
Note further that $\varphi_*$ is coarsely order preserving by the standard bound for quasiconformality constants of a composition, i.e. $\varphi_* \in \mathrm{CAut}\mathcal{QC}^\dagger(S)$.
This gives a natural homomorphism from $\mathcal{Q}(S)$ to $\mathrm{CAut}\mathcal{QC}^\dagger(S)$.
We will prove:

\begin{theorem}\label{thm-first-precise}
    Let $S = S_g$ with $g \geq 2$. Then the natural map $\mathcal{Q}(S) \to \mathrm{CAut}\mathcal{QC}^\dagger(S)$ is an isomorphism.
\end{theorem}

Note Theorem \ref{thm-first-precise} implies that the coarse order preservation of $\Phi$ is independent of the reference hyperbolic metric.
Since quasiconformal homeomorphisms $\varphi$ of $S$ induce automorphisms $\varphi_*$ that are coarsely order preserving, we obtain:

\begin{corollary}\label{cor-weakly-coarse-is-coarse}
    An automorphism $\Phi \in \mathrm{Aut}(\qcg(S))$ is coarsely order preserving if and only if $\Phi$ weakly coarse order preserving. 
\end{corollary}

Theorem \ref{thm-first} then follows immediately from Corollary \ref{cor-weakly-coarse-is-coarse} and Theorem \ref{thm-first-precise}.

\begin{remark}\label{rmk-coarse-order}
    There are two reasons why we prefer to work with weak coarse order preservation in the following.
    First, it is weaker than coarse order preservation, and so Theorem \ref{thm-first-precise} is stronger than if we only worked with coarse order preservation.
    Second, changing metrics is easily seen to be weakly coarse order preserving, but it is much less clear whether or not the involved constants $K'(K,\beta)$ can be taken to be $\beta$-independent.
    So the notion of weak coarse order preservation has desirable naturality properties in our setting.
\end{remark}

After the above discussion, what is left to show is:
\begin{proposition}\label{prop-is-quasiconf}
    Let $S = S_g$ with $g \geq 2$ and let $\Phi \in \mathrm{CAut}\mathcal{QC}^\dagger(S)$. Then there exists $\varphi \in \mathcal{Q}(S)$ so that $\Phi = \varphi_*$.
\end{proposition}

The strategy to prove Proposition~\ref{prop-is-quasiconf} follows similar lines to other characterizations of automorphisms of fine curve graphs \cite{booth2024automorphisms,long2023automorphisms}.
In \S\ref{Section:TopologyToGraphs}, we  identify graph structures that uniquely determine inessential quasicircles in the surface.
Then in \S\ref{Section:BuildQCFromGraph}, we use this graph structure to define a map from $\Aut \qcg (S)$ to automorphisms of an extended graph that also includes vertices for inessential quasicircles.
From there, we use inessential quasicircles to identify points and build the desired map on the surface.

We conclude by remarking that for punctured surfaces the analogue of Theorem \ref{thm-first-precise} fails without the assumption of weak coarse order preservation:

\begin{example}\label{ex-counter-ex}
    Let $\varphi: (0,1) \to (0,1)$ be an orientation-preserving homeomorphism and let $F : (\bbC - \{0\}) \to (\bbD - \{0\})$ be given by $(r,\theta) \mapsto (\varphi(r),\theta)$.
    Then every quasicircle in $\bbC -\{0\}$ is mapped by $\varphi$ to a quasicircle from smoothness of $F$.
    
    On the other hand, the annulus $A_r$ with inner radius $r$ and outer radius $1$ has modulus $-\log r$ and $F(A_r)$ has modulus $-\log \varphi(r)$.
    So for $\varphi$ with sufficiently quickly decaying derivative at $0$, the map $F$ unboundedly distorts the moduli of the annuli $A_r$ for small $r$ and so is not quasiconformal.
\end{example}

\section{Encoding Topological Data Into Graphs}\label{Section:TopologyToGraphs}
In this section, we find graph structures that we can use to identify inessential quasicircles in our surface. We start in \S\ref{Subsect:torus} by identifying pairs of quasicircles with simple intersections, called torus pairs. Then in \S\ref{Subsect:bigon}, we use torus pairs to identify a construction called bigon pairs that determine a unique inessential quasicircle.
Next in \S\ref{Subsect:sharing}, we define and prove a couple results about sharing pairs, which are pairs of bigon pairs that share the same inessential quasicircle.
Finally, we prove that associated arc graphs are connected in \S\ref{Subsect:ArcGraphs}, which shows that these inessential quasicircles are well-defined under automorphisms of the graph.

In this section, $S$ is a closed oriented surface of genus $g \geq 2$ unless otherwise specified.

\subsection{Torus pairs}\label{Subsect:torus}
In this subsection, we show that automorphisms of $\qcg(S)$ preserve the sets of nondegenerate transverse torus pairs, nondegenerate nontransverse torus pairs, and degenerate torus pairs. Additionally, the graph differentiates between these three different types of torus pairs.

We start with relevant definitions. A \textit{multiquasicircle} is a finite collection of pairwise disjoint essential  quasicircles in $S$.
This implies that any multiquasicircle forms a finite clique, i.e.~ a complete graph, in $\qcg(S)$. 
A multiquasicircle in $S$ is \textit{separating} if its complement has more than one component. 
Two quasicircles $a$ and $b$ \textit{lie on the same side} of a separating multiquasicircle $m$ if they are disjoint from $m$ and lie in the same complementary component.
We say that a graph is a \textit{join} if we can partition the set of vertices into two or more nonempty sets in such a way that every vertex from one set is connected by an edge to every vertex in the other sets. 
Also, the \textit{link} of a set $A$ of vertices in a graph is the subgraph spanned by the set of vertices that are not in $A$ and are connected by an edge to each vertex in $A$.

In the case of the fine curve graph, a curve is separating if and only if the link of the curve is a join. Similarly, we have the following:

\begin{lemma}\label{Lem:SeparatingLinkJoin}
     Let $S$ be an oriented surface. Let $a \in \qcg(S)$. Then $a$ is separating if and only
    if $\mathrm{link}(a)$ is a join.
    Similarly, if $ a = \{a_1, \ldots, a_n\} \in \qcg(S)$ is a multiquasicircle, then $a$ is separating if and only
    if $\mathrm{link}(a)$ is a join.
\end{lemma}

In~\cite{long2023automorphisms}, it is shown that automorphisms of the fine curve graph preserve the set of separating curves, the set of separating multicurves, and the sides of a separating multicurve.
The result below immediately follows from Lemma~\ref{Lem:SeparatingLinkJoin}.

\begin{lemma}\label{Lem:MulticurvesPreserves}
    Let $\Phi \in \Aut \qcg(S)$. Then $\Phi$ preserves the set
of separating quasicircles in $\qcg(S)$ and also preserves the set of separating multiquasicircles in $\qcg(S)$.
Moreover, $\Phi$ preserves the sides of a separating multiquasicircle, that is, $a$ and $b$ lie on the same
side of $m$ if and only if $\Phi(a)$ and $\Phi(b)$ lie on the same side of $\Phi(m)$.
\end{lemma}

\begin{definition}[Hull of quasicircles]
    The {\rm{hull}} in $S$ of a collection of quasicircles $\{\gamma_i\}$ is the union of the quasicircles
    along with any components of the complement of $\cup \gamma_i$ that are
    topologically embedded disks.
\end{definition}

The following lemma implies that hulls are preserved by automorphisms of $\qcg(S)$.
The proof of this lemma follows the proof of Lemma 2.4 from \cite{long2023automorphisms}, but considering quasicircles in place of curves in surfaces.

\begin{lemma}\label{Lem:AutomorphismsHullPreserved}
    Let $\Phi \in \Aut\qcg(S)$. If $A$ is a finite set of vertices of $\qcg(S)$ and a vertex $d$ lies in the hull of $A$, then $\Phi(d)$ lies in the hull of $\Phi(A)$.
\end{lemma}

\begin{definition}[Torus pair]
    Let $c,d \in \qcg(S)$. The pair $\{c,d\}$ is a {\rm{torus pair}} if $c \cap d$ is a single interval and $c$ and $d$ cross at that interval. A torus pair is said to be {\rm{degenerate}} if $c \cap d$ is a single point, and {\rm{nondegenerate}} otherwise.
\end{definition}

\begin{definition}[Nondegenerate transverse torus pair]
    A nondegenerate torus pair $\{a,b\}$ is said to be transverse if and only if there exists a quasicircle $c$ in the hull which does not equal either $a$ or $b$.
    A nondegenerate torus pair is said to be nontransverse otherwise.
\end{definition}

\begin{figure}[h]
    \includegraphics[width=2in]{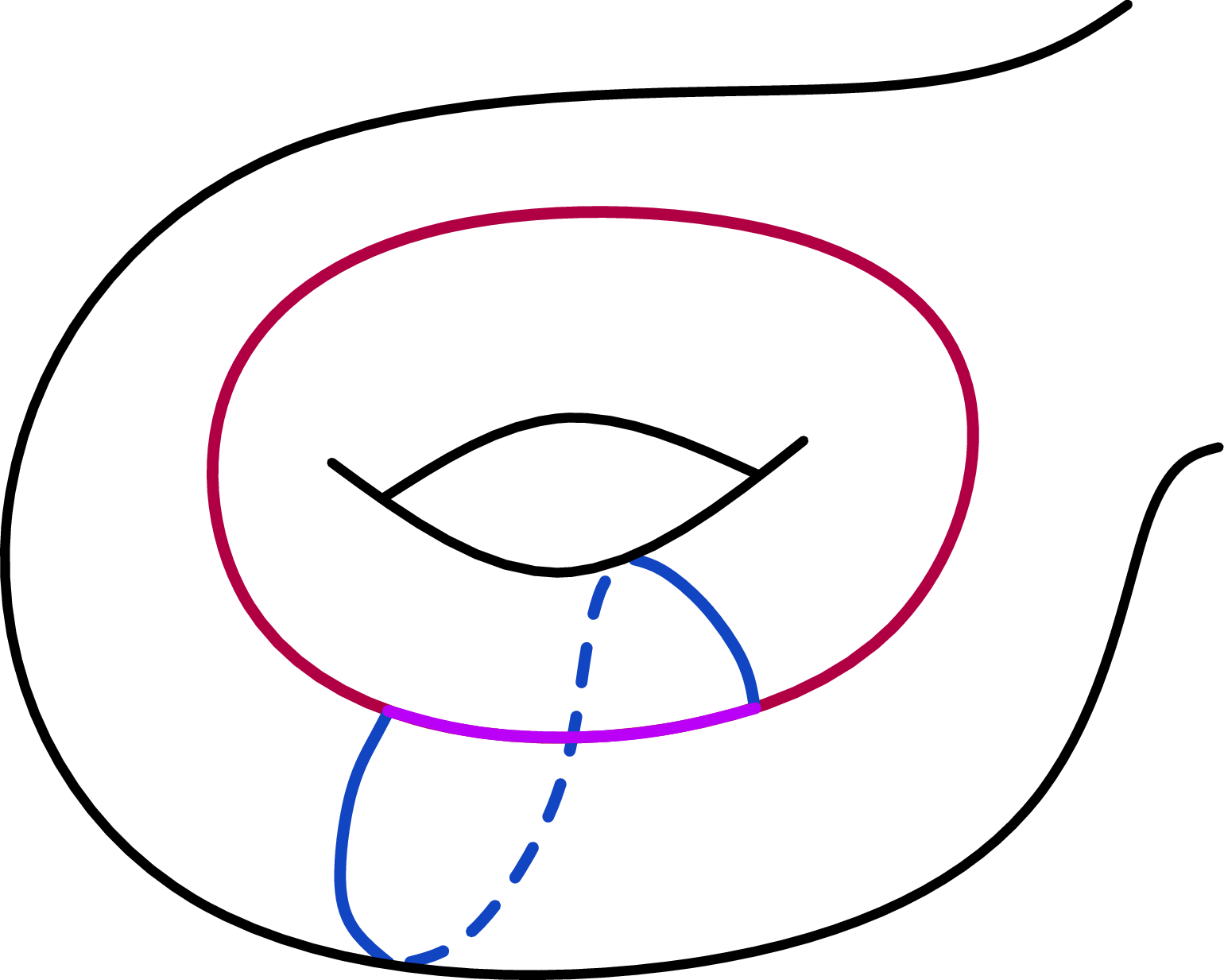}
    \hspace{.5in}
    \includegraphics[width=2in]{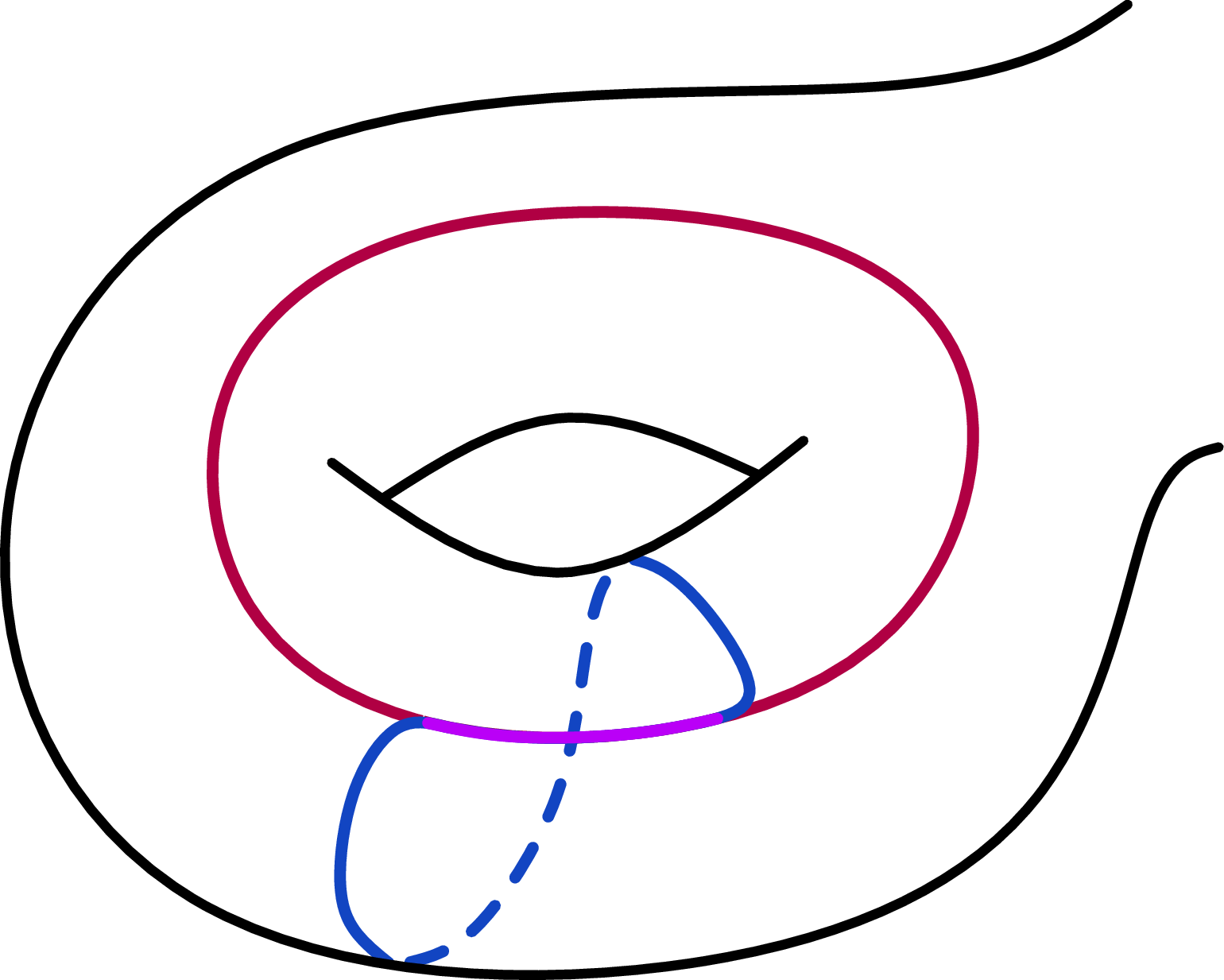}
    \caption{Two nondegenerate torus pairs are formed by the above blue and red curves. The intersections between the curves in each torus pair are drawn in purple.
    The pair on the left is transverse, and the third quasicircle is the symmetric difference between the curves.
    The pair on the right is nontransverse due to the cusps in the symmetric difference.}
\end{figure}

The following lemma allows us to distinguish when a pair of vertices corresponds to a torus pair in $\qcg(S)$.
We note that the forward direction follows by the definition of torus pairs, while the backward direction follows directly by replacing the $C^k$-curves in work of the first named author \cite{booth2024automorphisms} with quasicircles and applying \cite[Lemma 3.14]{booth2024automorphisms}.

\begin{lemma}\label{Lem:TorusPairClassification}
    Let $a, b$ be nonseparating, nonhomotopic quasicircles.
    Then $\{a,b\}$ is a torus pair if and only if the hull of $\{a,b\}$ contains at most one other vertex of  $\qcg(S)$ and there exists a separating quasicircle $\gamma \in \qcg(S)$ disjoint from $a$ and $b$ such that one side of $\gamma$ is a one-holed torus subsurface that contains $a$ and $b$.
\end{lemma}

Lemma \ref{Lem:TorusPairClassification} together with the definition of nondegenerate transverse torus pairs immediately provides us with a graph theoretical interpretation of when two quasicircles form a nondegenerate transverse torus pair.

\begin{lemma}\label{Lem:NonDegenNonTransTorusPairClassification}
    Let $a, b \in \qcg(S)$.
    Then $\{a,b\}$ is a nondegenerate transverse torus pair if and only if the hull of $\{a,b\}$ contains exactly three vertices of  $\qcg(S)$ and there exists a separating quasicircle $\gamma \in \qcg(S)$ disjoint from $a$ and $b$ such that one side of $\gamma$ is a one-holed torus subsurface which contains $a$ and $b$.
\end{lemma}

We now provide the graph structure which is unique to nondegenerate torus pairs.
We note that the proof from Lemma 3.15 of~\cite{booth2024automorphisms} for the same result holds by considering quasicircles in place of $C^k$-curves.

\begin{lemma}
   Let $\{a,b\}$ be a torus pair in $\qcg(S)$ whose interval of intersection is denoted $c$.
   Then $\{a,b\}$ is a nondegenerate torus pair if and only if there exist disjoint homotopic curves $d_1$ and $d_2$ in $\qcg(S)$ such that $\{a,d_1\}, \{a,d_2\},\{b,d_1\}$, and $\{b,d_2\}$ are all torus pairs and for any quasicircle $e \in \qcg(S)$ disjoint and non-homotopic to $d_1$ and $d_2$, then $e$ intersects $a$ if and only if $e$ intersects $b$.
\end{lemma}

Lemma 3.4 of \cite{booth2024automorphisms} holds for curves in general, so we know that for two disjoint quasicircles $a,b \in \qcg(S)$, $a$ and $b$ are homotopic if and only if $\mathrm{link}(a,b)$ is a join where one part contains only separating or only nonseparating curves.

We now show that two quasicircles $a,b \in \qcg(S)$ are homotopic if and only if there exists a finite path of disjoint homotopic curves in $\qcg(S)$.

\begin{lemma}\label{Lem:QuasiCirclesHomotopic}
    Let $a, b \in \qcg(S)$. Then $a$ and $b$ are homotopic if and only if there
    exists a finite path $a_1, ..., a_n$ of homotopic curves in $\qcg(S)$ with $a = a_1,$ $b = a_n$ and $a_i$ disjoint from $a_{i+1}$ for $i = 1, ..., n-1$.
\end{lemma}

\begin{proof}
    Let $a, b \in \qcg(S)$ be homotopic, and let $H : S^1 \times [0,1] \to S$ be an isotopy from $a$ to $b$.
    For $t \in [0,1]$, let $H_t(a)$ be the image of $S^1 \times \{t\}$ so that $H_0(a) = a$ and $H_1(a) = b$.
    Notice that for all $t \in (0,1)$, it may not be that $H_t(a)$ is a quasicircle.
    For each vertex of $c \in \qcg(S)$, we define,
    \[
        I_c = \{t \in [0,1] : a_t \text{ and } c \text{ are contained in the interior of some annulus}\}
    \]
    Each $I_c$ is open, and since $[0,1]$ is compact, we may find a sequence of open intervals $I_0, \ldots, I_k$ which cover $[0,1]$ so that each $I_i$ is in some $I_{c_i}$.
    Set $t_0 = 0$ and  $t_{k+1} = 1$.  We may suppose that $I_i \cap I_{i+1} \neq \emptyset$,
    so let $t_i \in I_i \cap I_{i+1}$.
    Then each pair $\{H_{t_i}(a) , H_{t_{i+1}}(a)\}$ is contained inside of some annulus.
    There exist quasicircles $a_{i}, a_{i+1}$, with $a_0=a$ and $a_{k+1}=b$, within some small neighborhood of $\{H_{t_i}(a), H_{t_{i+1}}(a)\}$ that are also contained in some annulus, as well. Let $d_i$ denote one of the boundary components from this annulus, and note that again we may find a quasicircle, which we also denote by $d_i$, within a small neighborhood of the boundary and is disjoint from both $a_i$ and $a_{i+1}$.
    Thus the quasicircle $d_i$ is disjoint and homotopic to both $a_{i}$ and $a_{{i+1}}$, and the sequence $\{a, d_0,a_1, \ldots a_{k}, d_{k}, b\}$ is a path in $\qcg(S)$.
\end{proof}

The following lemma  distinguishes in the quasicircle graph which separating curves bound a one-holed torus.
The proof follows immediately from Lemma 3.10 of \cite{booth2024automorphisms}.

\begin{lemma}\label{Lem:CharacterizationSepCurveOneHoledTori}
    Let $\gamma \in \qcg(S)$ be a separating curve. Then $\gamma$ bounds
    a one-holed torus
    subsurface if and only if one component of $S \setminus \gamma$ only contains curves that are either
    homotopic to $\gamma$ or nonseparating.
\end{lemma}

The following lemma states that the set of quasicircles which bound one-holed tori is preserved by automorphisms of the quasicircle graph.
The following lemma is a weaker version of Lemma 3.11 from ~\cite{booth2024automorphisms}, but is sufficient for our current purposes.

\begin{lemma}
    Let $\Phi \in \Aut \qcg(S)$. Then $\Phi$ preserves the set of quasicircles that bound one-holed tori in $S$.
\end{lemma}

\begin{proof}
    Lemma \ref{Lem:CharacterizationSepCurveOneHoledTori} states that curves $\gamma$ which separate one holed tori are characterized by either quasicircles homotopic to $\gamma$ or by non-separating quasicircles.
    We know automorphisms of $\qcg(S)$ preserve separating quasicircles and homotopic quasicircles by Lemmas \ref{Lem:SeparatingLinkJoin} and  \ref{Lem:QuasiCirclesHomotopic}, which proves our claim.
\end{proof}

The preceding results in this section provide us with the following proposition:

\begin{proposition}\label{prop-toruspairs}
    Let $\Phi \in \Aut \qcg(S)$. Then $\Phi$ preserves the set of degenerate torus pairs, the set of nondegenerate nontransverse torus pairs, and the set of nondegenerate transverse torus pairs.
\end{proposition}

\subsection{Bigon pairs}\label{Subsect:bigon}
 A pair of homotopic vertices $c,d \in \qcg(S)$ are said to form a \textit{bigon pair} if $c \cap d$ is a nontrivial closed interval and the boundary of the disk component of $S \setminus \{c,d\}$  forms a quasicircle in $S$.
 In this section, our goal is to prove that bigon pairs are preserved under automorphisms of $\qcg(S)$.

\begin{lemma}[Existence of Bigon Pairs]\label{Lem:ExistenceBigonPairs}
    Let $\gamma$ be a quasicircle in $S$ with distinct points $p_1, p_2 \in \gamma$. Let $\gamma_1, \gamma_2$ be the closures of the connected components of $\gamma - \{p_1, p_2\}$.
    Then there exist essential quasicircles $\eta_1, \eta_2$ in $S$ so that:
    \begin{enumerate}
        \item $\eta_i \cap \gamma = \gamma_i$ ($i = 1, 2$),
        \item $\eta_1 \cap \eta_2 = \eta_1 - \text{int}(\gamma_1) = \eta_2 - \text{int}(\gamma_2)$.
    \end{enumerate}
\end{lemma}

\begin{proof}
    The proof uses the same technique appearing in Remark \ref{remark-def-well-behaved} in \S\ref{s-qcbasics}. 
    Take an annular neighborhood $U$ of $\gamma$.
    Then there is a round annulus $A$ containing the unit circle $\gamma_0$ in $\bbC$ and a conformal homeomorphism $\psi : U  \to A$.
    Theorem \ref{thm-smooth-away-from-bad} then gives a quasiconformal homeomorphism of $\bbC$ that so $\varphi(\gamma) = \gamma_0$.
    There is then a neighborhood $U' \subset U$ so that $\psi^{-1} \circ \varphi \circ \psi : U' \to S$ is quasiconformal, is smooth away from $\gamma$ and so that $\gamma' = \psi^{-1} \circ \varphi \circ \psi(\gamma)$ is smooth.
    After potentially shrinking $U'$ again, extend $\psi^{-1} \circ \varphi \circ \psi$ to a homeomorphism $h : S \to S$ that is smooth away from $\gamma$ and hence quasiconformal on $S$ from the differential characterization of quasiconformality.
    The desired construction can now be run without trouble with the smooth curve $\gamma'$ so that the curves $\eta_1', \eta_2'$ are smooth away from right-angle turns at $h(p_i)$, and hence are quasicircles.
    Then $\eta_i = h^{-1}(\eta_i')$ are quasicircles and have the desired properties.
\end{proof} 

\subsection*{Annulus sets}
Suppose that $(a, b)$ is an ordered pair of disjoint, homotopic vertices of $\qcg(S)$.
As $g \geq 2$, there is a unique annulus $A$ in $S=S_g$ whose boundary is $a \cup b$. Let $\qcg(a, b)$ be
the set of vertices of $\qcg(S)$ contained in $A$. We refer to $\qcg(a, b)$ as an \textit{annulus set}. We say
that a pair of vertices of $\qcg(S)$ is an \textit{annulus pair} if they lie in some $\qcg(A')$, in other words, the quasicircle graph restricted to some annulus $A'$. 
A \textit{nonseparating 
noncrossing annulus pair} is an annulus pair where both curves are nonseparating and the
pair has no crossing intersections.
There is a natural partial ordering on the annulus set $\qcg(a,b)$: we say that $c \preceq d$ if $c$ and $d$ are
noncrossing and each component of $c \setminus d$ lies in a component of $A \setminus d$ bounded by $a$.

The following lemma is the analogue of Lemma 2.6 from \cite{long2023automorphisms}:

\begin{lemma}\label{Lem:PreservesAnnulusSets}
    Let $\Phi \in \Aut\qcg(S)$ and  let $a$,$b$ be
    disjoint, homotopic nonseparating quasicircles. Then:
    \begin{enumerate}
        \item The quasicircles $\Phi(a)$ and $\Phi(b)$ are disjoint, homotopic nonseparating quasicircles.
        \item The image of $\qcg(a,b)$ under $\Phi$ is $\qcg(\Phi(a),\Phi(b))$
        \item If $c,d \in \qcg(a,b)$ are noncrossing then $\Phi(c)$ and $\Phi(d)$ are noncrossing.
        \item If $c \preceq d$ in $\qcg(a,b)$ then $\Phi(c) \preceq \Phi(d)$ in $\qcg(\Phi(a),\Phi(b))$.
    \end{enumerate}
\end{lemma}

\begin{proof}
    The proof of this lemma follows from Lemma \ref{Lem:MulticurvesPreserves} and the following two observations.
    First, note that two disjoint nonseparating curves a and $b$ in $S$ are homotopic if and only if both $a$ and $b$ form a separating multiquasicircle and all separating quasicircles disjoint from both $a$ and $b$ lie on the same side of the multiquasicircle $a \cup b$.
    Second, noncrossing quasicircles $c,d \in \qcg(a,b)$ satisfy $c \preceq d$ if and only if there is an element of $\qcg(a,b)$ which intersects $a$ and $c$ but not $d$.
    Then the claims follow in order.
\end{proof}

\subsection*{Type 1 and type 2 curves}
Suppose that $\{c, d\}$ is a nonseparating noncrossing annulus pair,
and suppose that $e$ is a quasicircle so that $\{c, e\}$ and $\{d, e\}$ are degenerate torus pairs. If $c\cap e$ and
$d \cap e$ are the same point, then we say that $e$ is a \textit{type 1 quasicircle} for $\{c, d\}$. Otherwise we say
that $e$ is a \textit{type 2 quasicircle} for $\{c, d\}$.
\begin{figure}[h]
    \begin{tikzpicture}
        \small
        \node[anchor=south west, inner sep = 0] at (0,0){\includegraphics[width=3.5in]{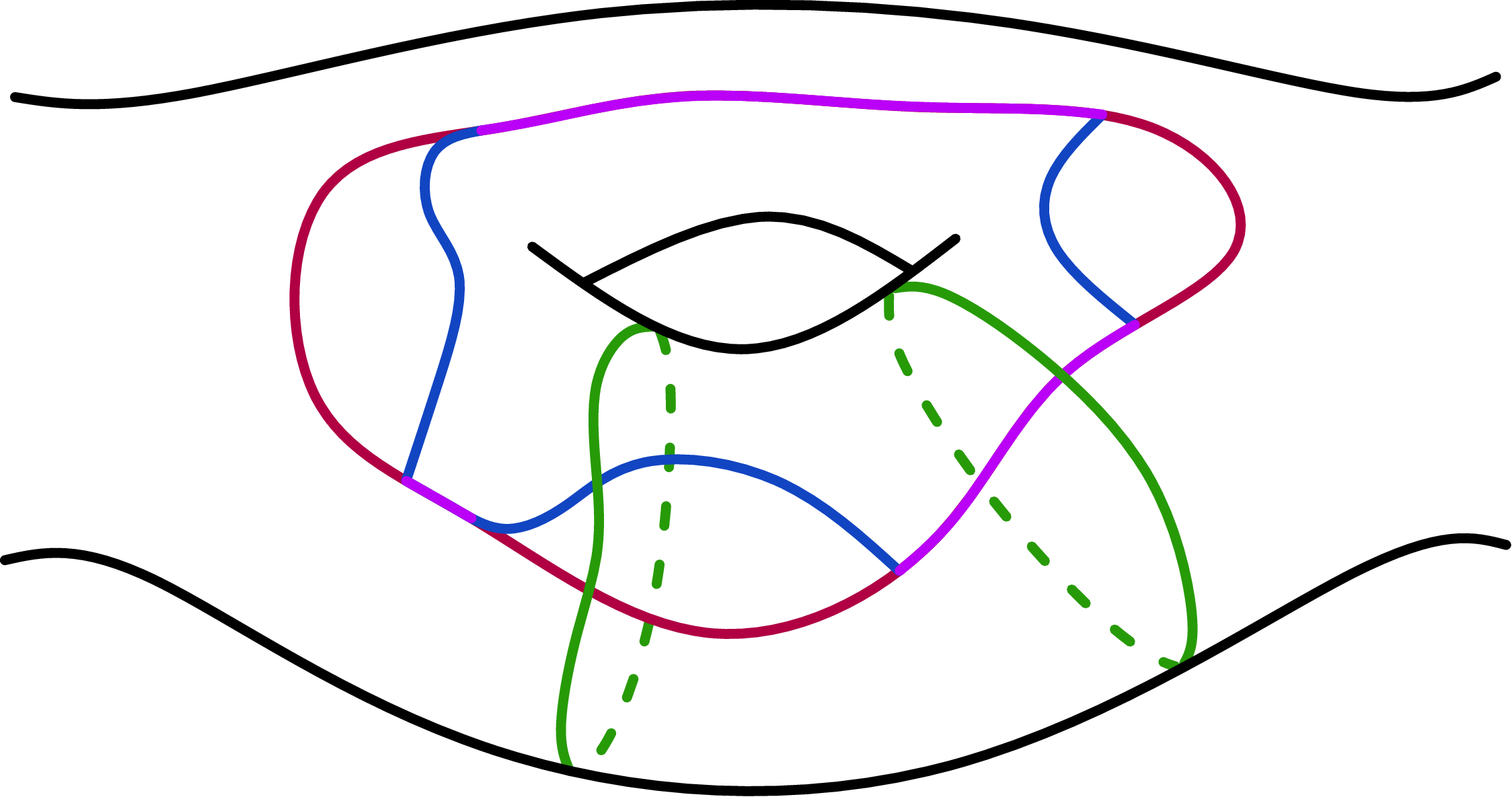} };
        \node at (7.5,1.6) {\color{OliveGreen}type 1};
        \node at (2.75,.8) {\color{OliveGreen}type 2};
        \node at (2.75,3.5) {\color{blue}$d$};
        \node at (1.55,3) {\color{Maroon}$c$};
    \end{tikzpicture}
    \caption{Examples of type 1 and type 2 quasicircles for the nonseparating noncrossing annulus pair $\{c, d\}$.}
\end{figure}
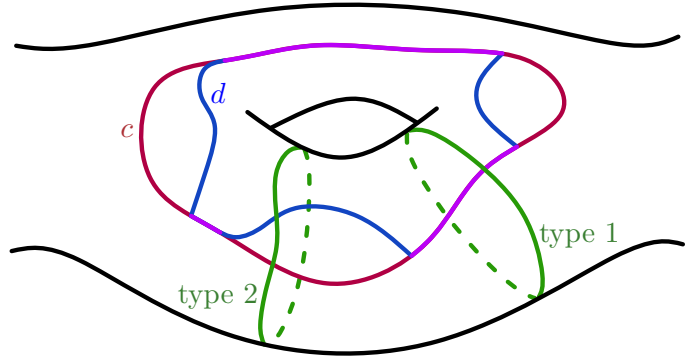

\begin{lemma}\label{Lem:Type1Type2Preserved}
    Let $\Phi \in \Aut \qcg(S)$. Then $\Phi$ preserves type 1
and type 2 quasicircles for nonseparating noncrossing annulus pairs. More precisely, if $\{c, d\}$ is a
nonseparating noncrossing annulus pair and $e$ is a type 1 quasicircle for $\{c, d\}$, then $\Phi(e)$ is a type
1 quasicircle for the nonseparating noncrossing annulus pair $\{\Phi(c), \Phi(d)\}$, and similarly for type
2 quasicircles.
\end{lemma}

\begin{proof}[Proof Sketch]
Since this proof is nearly identical to Lemma 2.7 of \cite{long2023automorphisms}, we only give a sketch of the arguments here. 
    First, if $c$ and $d$ are disjoint, then any quasicircle $e$ with $\{c,e\}$ and $\{d,e\}$ being degenerate torus pairs must be a type 2 curve. Since $\Phi$ preserves disjoint curves and preserves degenerate torus pairs, we have that $\Phi$ preserves that $e$ is a type 2 quasicircle.

    Now if $c,d \in \qcg(a,b)$ are not disjoint, then they have at least one point, or interval, of intersection.
    In this case, $e$ is a type 2 quasicircle for $\{c,d\}$ if and only if there is a quasicircle $f$ with the following properties: 
    \begin{enumerate}
        \item $f$ is contained in the hull of $\{c,d,e\}$,
        \item $f$ is not contained in $\qcg(a,b)$, and
        \item $f$ is not equal to $e$.
    \end{enumerate}
    The forward direction can be proved directly by pushing off of $e$ within an open disk bounded by $c\cup d$. For the other
    direction, if $e$ is of type 1, then any quasicircle that satisfies the first
    two given properties would have to contain all of $e$, and hence would fail the third property. 
\end{proof}

We can now show the following:

\begin{proposition}\label{Prop:BigonPairsPreserved}
    Let $\Phi \in \Aut\qcg(S)$. Then $\Phi$ preserves the set of nonseparating bigon pairs.
\end{proposition}

    The proof of the corresponding proposition in \cite{long2023automorphisms} (see [Proposition 2.1] only depends on automorphisms of the fine curve graph preserving annulus sets, type 1 curves, and type 2 curves.
As we have the corresponding results in the case of quasicircles, the proof from \cite{long2023automorphisms} extends to $\qcg(S)$. 

\subsection{Sharing pairs}\label{Subsect:sharing}
We denote by $S_g^b$ the surface given by removing the interiors of $b$ disjoint disks with quasicircle boundaries from $S_g$. 
We say that a pair of bigon pairs $\{\{a, b\},\{a',b'\}\}$  is a
\textit{sharing pair} if the nondisk $S^1_g$ components of $S \setminus \{a\Delta b\}$ and $S \setminus \{a'\Delta b'\}$ are the same  and the  arcs $a \cap b$, $a' \cap b'$ in $S^1_g$ have disjoint interiors with four unique endpoints on the boundary quasicircle. A sharing pair is said to be \textit{linked} if the endpoints of $a \cap b$ separate the endpoints of $a' \cap b'$ in the boundary. 
In this section, we prove that automorphisms of $\qcg(S)$ preserve linked sharing pairs.

\begin{figure}[h]
    \includegraphics[width=2.5in]{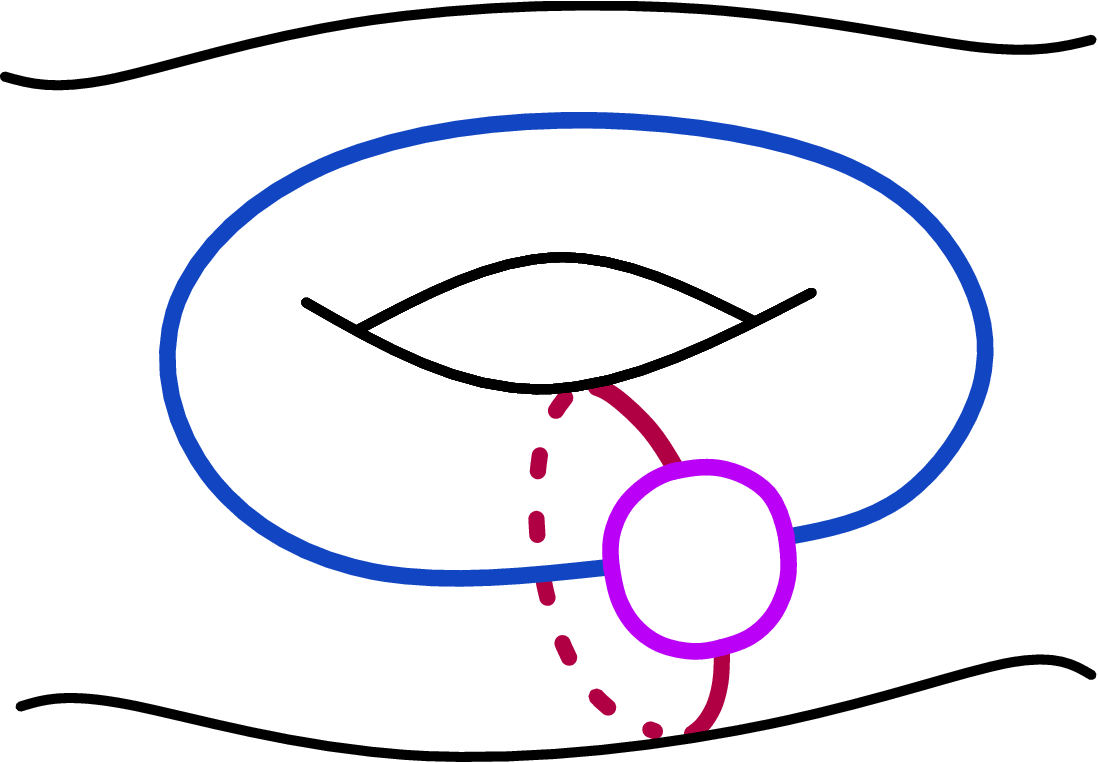}
    \caption{A linked sharing pair.}
\end{figure}

\begin{definition}[Torus triple]
    If $\{c, d\}$ is a  nondegenerate transverse torus pair in $S$, then there is exactly one
    other essential quasicircle $e$ contained in $c \cup d$; the quasicircle $e$ is the closure in $S$ of the symmetric
    difference $c\Delta d$. We also refer to $\{c, d, e\}$ as a {\rm{torus triple}}, since any two
    elements of the triple form a torus pair determining the third.
\end{definition} 

\begin{lemma}\label{Lem:TorusTriplesPreserved}
    Let $\Phi \in \Aut \qcg(S)$. Then $\Phi$ preserves the
    set of torus triples.
\end{lemma}

\begin{proof}
    Suppose $\{c, d, e\}$ is a torus triple. Then $e$ is the unique vertex (other than $c$ and
    $d$) contained in the hull of $\{c, d\}$. By Lemma \ref{Lem:AutomorphismsHullPreserved}, $e$ is the unique quasicircle whose link contains
    the link of $\{c, d\}$. Since nondegenerate transverse torus pairs are preserved by Proposition~\ref{prop-toruspairs}, it now follows that torus triples are
    preserved.
\end{proof}

\begin{figure}[h]
    \includegraphics[width=1.5in]{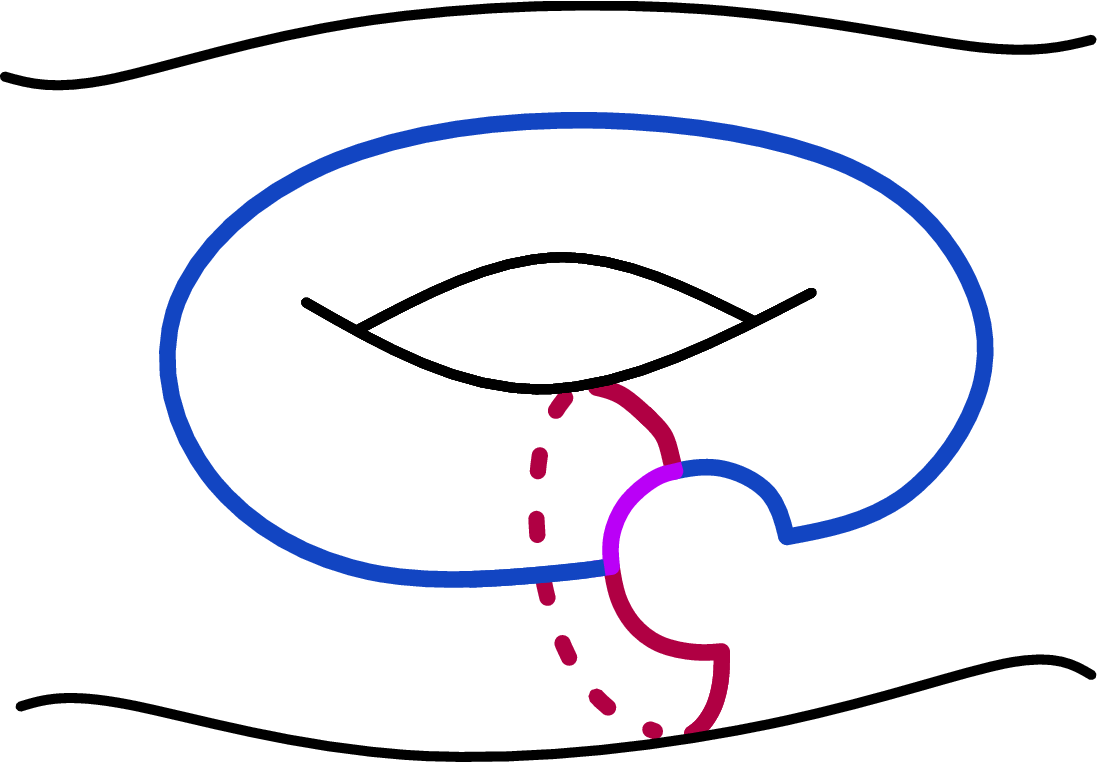}
    \hspace{.25in}
    \includegraphics[width=1.5in]{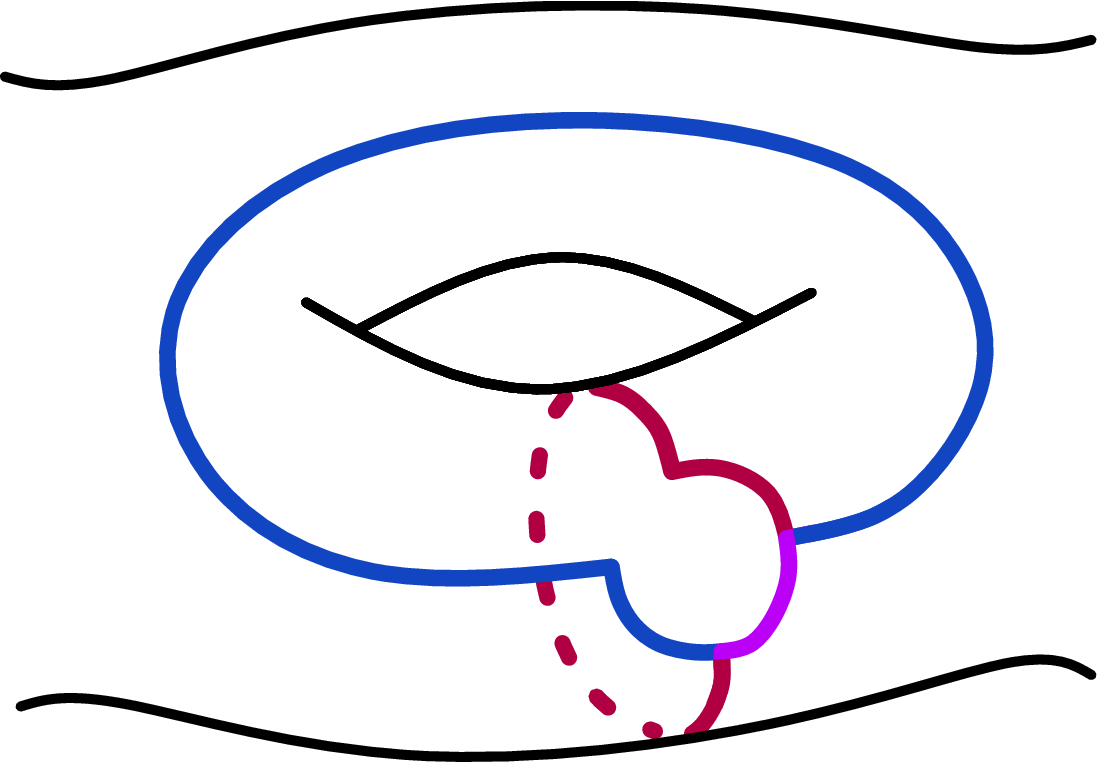}
    \hspace{.25in}
    \includegraphics[width=1.5in]{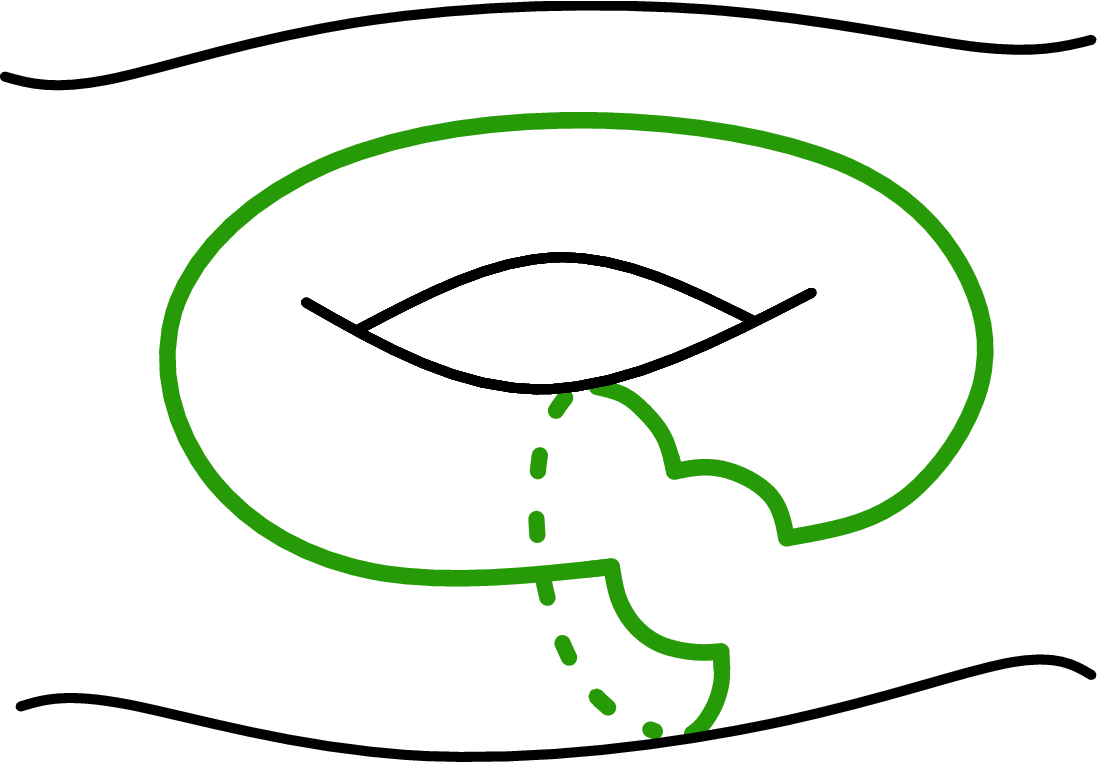}
    \caption{The torus pairs $\{c, d'\}$ and $\{c', d\}$ and the quasicircle that forms a torus triple with both.}
    \label{fig:torustriple1}
\end{figure}

\begin{proposition}\label{Prop:SharingPairsPreserved}
Let $\Phi \in \Aut \qcg(S)$. Then $\Phi$ preserves linked sharing pairs.
\end{proposition}

\begin{proof}
    We claim that bigon pairs $\{c,d\}$ and $\{c',d'\}$ form a linked sharing pair if and only if the following conditions hold:
    
    \begin{enumerate}
        \item each of $\{c, c'\}$, $\{c,d'\}$, $\{c', d\}$, and $\{d,d'\}$ form nondegenerate transverse torus pairs 
        \item there is a quasicircle that forms a torus triple with both $\{c,d'\}$ and $\{c',d\}$
        \item there is a quasicircle that forms a torus triple with both $\{c,c'\}$ and $\{d,d'\}$
    \end{enumerate}
    The proposition then follows from the claim, Proposition \ref{Prop:BigonPairsPreserved}, Lemma \ref{Lem:TorusPairClassification} and Lemma \ref{Lem:TorusTriplesPreserved}.

    For the forward direction, we can use the constructions indicated by Figures~\ref{fig:torustriple1} and ~\ref{fig:torustriple2}.
    There, the curve $e$, respectfully $e'$, forms a torus triple with $\{c,d'\}$ and $\{c',d\}$, respectfully $\{c,c'\}$ and $\{d,d'\}$.
    We note that this holds since $\{c,d\}$ and $\{c',d'\}$, respectfully $\{c,c'\}$ and $\{d,d'\}$, both form a quasicircle in $S$, the construction of $e$, respectfully $e'$, outlined is also a quasicircle for some quality constant $K$, respectfully $K'$, using the localization Lemma \ref{lemma-bt-localizes}.
    
\begin{figure}[h]
    \includegraphics[width=1.5in]{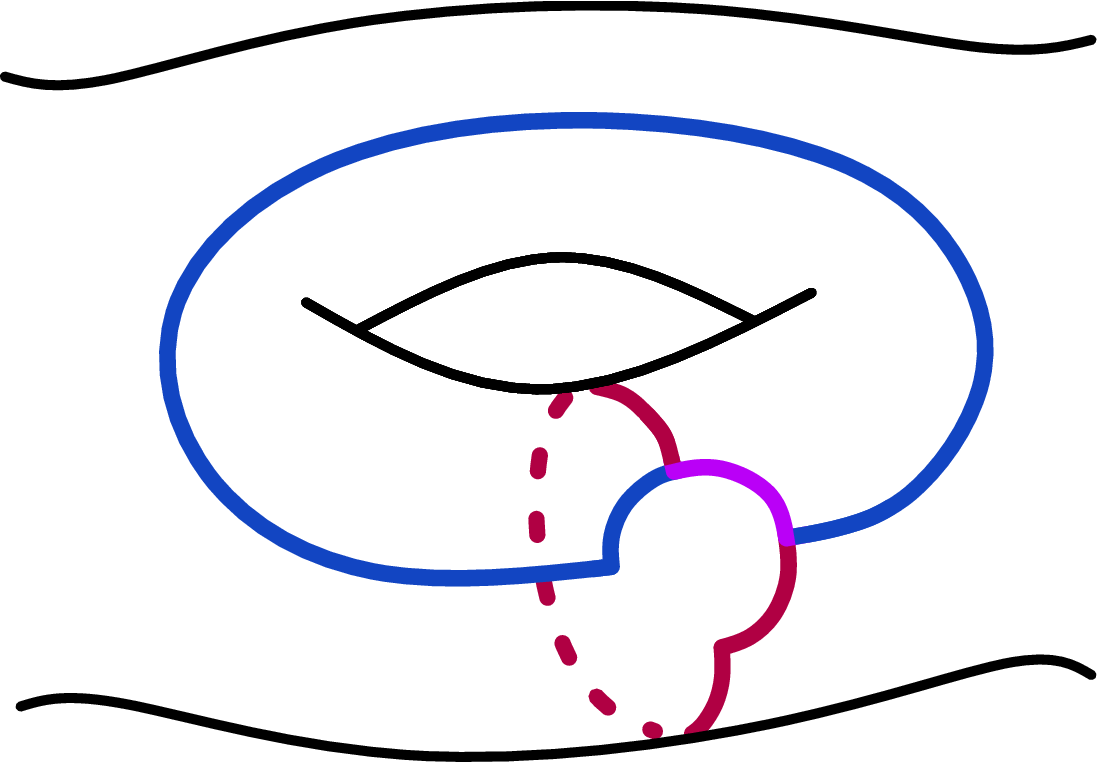}
    \hspace{.25in}
    \includegraphics[width=1.5in]{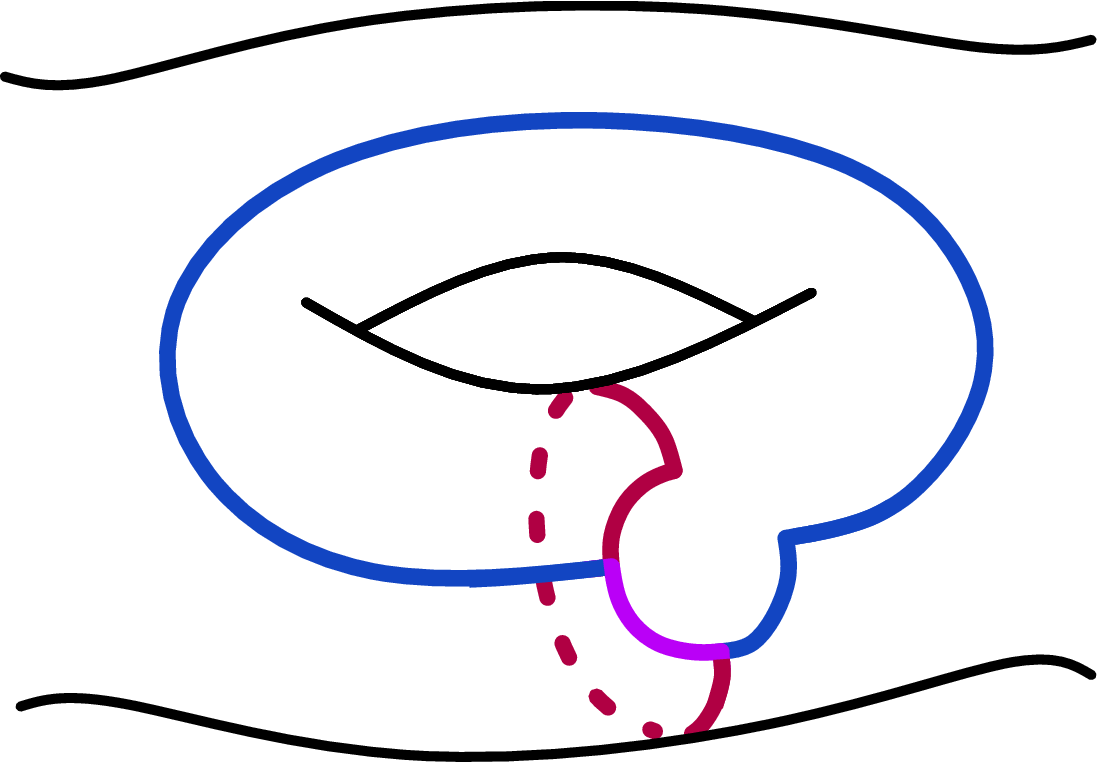}
    \hspace{.25in}
    \includegraphics[width=1.5in]{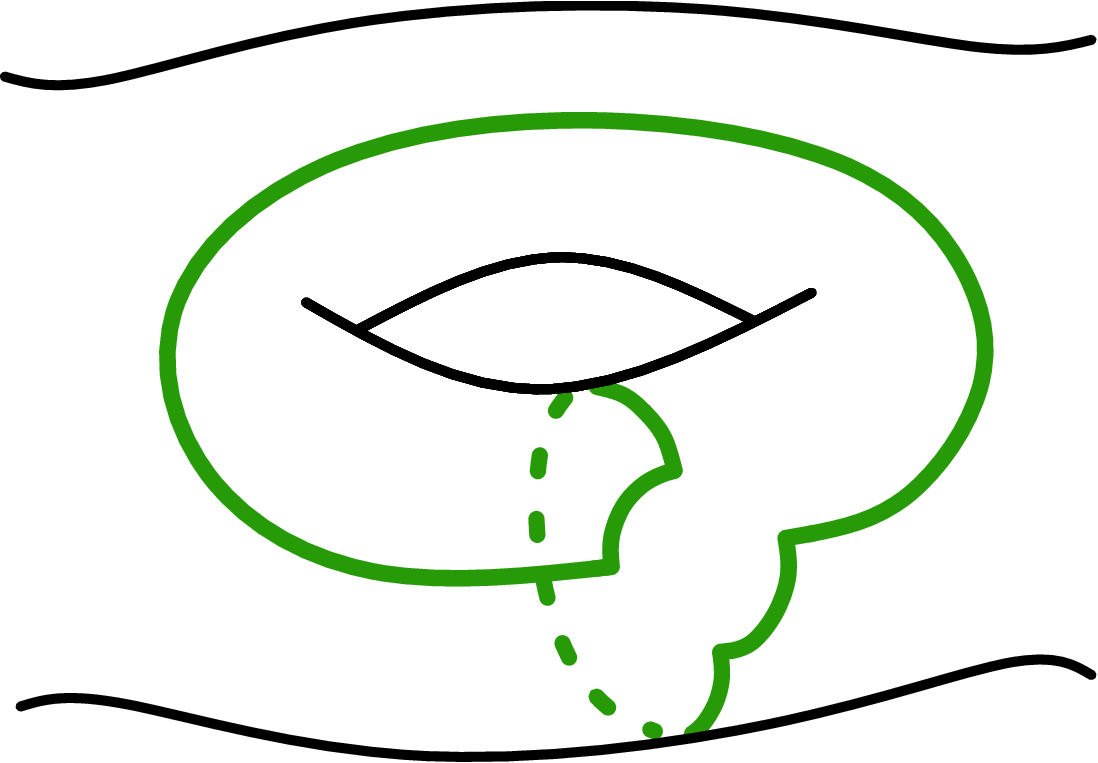}
    \caption{The torus pairs $\{c, c'\}$ and $\{d, d'\}$ and the quasicircle that forms a torus triple with both.}
    \label{fig:torustriple2}
\end{figure}

    For the other direction, we may precisely follow the proof of the analogous statement in ~\cite{long2023automorphisms} (see Proposition 2.2) to show that any quasicircles that satisfy the conditions in the claim form bigon pairs with the same inessential quasicircle. 

    If these constructed bigon pairs are not linked sharing pairs, then at least one of the pairs from condition (1) will either be disjoint or intersect in more than one interval and thus not be a nondegenerate transverse torus pair.
\end{proof}

\subsection{Connectivity of arc graphs}\label{Subsect:ArcGraphs}
We again consider surfaces $S^b_g$ with boundary that are obtained by cutting $S_g$ along a union of disjoint quasicircles. 
Our goal is to prove the linked quasi-arc graph $\lqag\left(S^b_g\right)$, defined below, is connected for $g\geq 1$ and $b>0$. The proofs in this section mirror similar results in \cite{long2023automorphisms}.

\subsection*{The quasi-arc graph}
We first define the quasi-arc graph, denoted $\qag(S)$.
The vertices of this graph are simple proper arcs in $S_g^b$ relative to the boundary, where the extension of the arc along either of the subarcs of the boundary gives an  essential quasicircle in the closed surface $S$. Two vertices are connected by an edge if the two arcs are disjoint.
Note that the same argument as in Lemma \ref{Lem:ExistenceBigonPairs} shows $\qag(S)$ is non-empty.

\begin{proposition}\label{Prop:QuasiArcGraphConnected}
    For any $S = S_g^b$ with $g\geq1$ and $b>0$, the quasi-arc graph $\qag(S)$ is connected.
\end{proposition}

\begin{proof}
    There is a natural surjection from $\qag(S)$ to the arc graph $\mathcal{A}(S)$ by taking isotopy classes.
    Now, $\mathcal{A}(S)$ is known to be connected when the surface has nontrivial genus, for instance because the arc complex is contractible (see \cite{hatcher1991arcgraph}).
    From the definition of $\mathcal{A}(S)$, we need only show that there is a path between any two homotopic vertices $a,b$ of $\qag(S)$.
    
    There exists an isotopy $H : I \times I \to S$ from $a$ to $b$.
    Let $c$ be a quasi-arc which is disjoint from $a$.
    By the isotopy extension property \cite[Corollary 1.4]{edwardsKirby1971deformations}, we can extend $H$ to be an isotopy on the entire surface, which implies $H$ will also be an isotopy from $c$ to some arc $d$, where $d$ may not be a quasi-arc.

    Note that for any $t\in I$, the arcs $H_{t}(a)$ and $H_{t}(c)$ might not be quasi-arcs, but they must be disjoint. Since $H_{t}(a)$ is contained in the open set $S\setminus H_{t}(c)$ and smooth arcs (and thus quasi-arcs) are dense away from the boundary, by Lemma~\ref{Lem:ExistenceBigonPairs} there exists a quasi-arc $a_t$ in some small neighborhood of $H_{t}(a)$ that is disjoint from $H_{t}(c)$. Similarly, we can find a quasi-arc $c_t$ in a small neighborhood of $H_{t}(c)$ that is disjoint from both $a_t$ and $H_{t}(a)$.

    For any $t \in I$ there is a neighborhood $U$ of $t$ so that the curve $H_{t'}(a)$ is disjoint from $c_t$ for any $t' \in U$.
    So by compactness, we may partition the interval $I$ into finitely many pieces $I = \{t_0 = 0, t_1, t_2, \ldots, t_k = 1\}$ such that $c_{t_i}$ and $H_{t_{i+1}}(a)$ are disjoint. Thus the quasi-arc $a_{t_{i+1}}$ can also be made disjoint from $c_{t_i}$. 

    Thus $$\{a=a_0, c=c_0, a_{t_1}, c_{t_1}, \ldots, a_{t_i}, c_{t_i}, a_{t_{i+1}}, \ldots, c_{t_{k-1}}, a_{t_k}=b\} $$ is a path in $\qag(S)$ between any arbitrary homotopic quasi-arcs $a$ and $b$. Thus $\qag(S)$ is connected.
\end{proof}

\subsection*{The nonseparating quasi-arc graph}
We say that a quasi-arc in a surface $S$ is nonseparating if its complement in $S$ is connected. The nonseparating quasi-arc graph, denoted $\nqag(S)$, is the subgraph of $\qag(S)$ spanned by the nonseparating quasi-arcs.

\begin{corollary}\label{Cor:NonSeparatingQuasiArcGraphConnected}
    For any $S = S_g^b$ with $g\geq1$ and $b>0$, the nonseparating quasi-arc graph $\nqag(S)$ is connected.
\end{corollary}

\begin{proof} 
    This can be seen by considering the neighbors of any separating vertex in a path in $\qag(S)$ between two nonseparating quasi-arcs. If these neighbors are on separate sides of the separating vertex, then they must be disjoint, and this vertex can be removed from the path. Otherwise, if the neighbors instead reside on the same side, then there exists a nonseparating quasi-arc on the other side that is disjoint from both. Thus the separating vertex can be replaced by a nonseparating vertex.
\end{proof}

\subsection*{The linked quasi-arc graph}
 Let $d_0$ be a distinguished component of $\partial S$.  We say that $a, b \in \qag(S)$ are \textit{linked} at $d_0$ if all endpoints of $a, b$ are on $d_0$ and the endpoints of $a$ separate the endpoints of $b$ in $d_0$.

We define the linked quasi-arc graph $\lqag(S,d_0)$ to be the graph whose vertices are nonseparating quasi-arcs in $S$ with both endpoints at $d_0$ and whose edges connect disjoint quasi-arcs that are linked at $d_0$. When convenient we suppress $d_0$ in the notation and write $\lqag(S)$.

\begin{corollary}\label{Cor:LinkedQuasiArcGraphConnected}
    For any $S = S_g^b$ with $g\geq1$ and $b>0$, the linked quasi-arc graph $\lqag(S)$ is connected.
\end{corollary}

\begin{proof}
    Since $\nqag(S)$ is connected by Corollary~\ref{Cor:NonSeparatingQuasiArcGraphConnected}, we only thing need to  check that there is a path between any two $a, b \in \lqag(S)$ that are not linked at $d_0$. Since $a$ and $b$ are both nonseparating, then $S\setminus \{a \cup b\}$ has at most two components. Notably, the component of $d_0$ between the endpoints of $a$ not containing the endpoints of $b$ and the component of $d_0$ between the endpoints of $b$ not containing the endpoints of $a$ must both be in the same component. So there is a quasi-arc $c$ with endpoints in each of these components of $d_0$, and hence is linked with both $a$ and $b$.
\end{proof}

\section{Building the quasiconformal map}\label{Section:BuildQCFromGraph}

This section has three parts. First, we use the graph information from \S\ref{Section:TopologyToGraphs} to extend an automorphism of our original quasicircle graph to an automorphism of the extended quasicircle graph. Second, we use the automorphism of the extended graph to build a homeomorphism of the surface. Finally, we show that this homeomorphism is also a quasiconformal map.

As in \S \ref{Section:TopologyToGraphs},  $S$ is a closed oriented surface of genus $g \geq 2$ unless otherwise specified.

\subsection{From essential to inessential}

Our goal is to define an extended version of the quasicircle graph which serves as a tool to prove that the group of automorphisms of the quasicircle graph embeds into $\mathrm{Homeo}(S)$.

\begin{definition}[Extended quasicircle graph]
    Let $\eqcg(S)$ be the graph with vertices
    all quasicircles in $S$ and edges between vertices corresponding to disjoint quasicircles.
\end{definition}

The goal of this subsection is to prove the following result:
\begin{proposition}\label{Prop:AutQuasiGraphToAutExtended}
    Let $\Phi \in \Aut \qcg(S)$. Then there exists a $\overline{\Phi} \in \Aut \eqcg(S)$ such that the restriction of $\overline{\Phi}$ to the essential vertices is exactly $\Phi$.
\end{proposition}

Before we can give a description of this extended automorphism, we must first  describe the edges of $\eqcg(S)$ using only information from $\qcg(S)$. This is accomplished in the following lemma.

\begin{lemma}\label{Lem:EdgeBetweenExtendedGraph}
Edges in $\eqcg(S)$ can be differentiated in $\qcg(S)$ by the following constructions:
    \begin{enumerate}
        \item Two inessential quasicircles $e$ and $f$ are disjoint and not nested if and only if there exist disjoint bigon pairs $\{c, d\}$ for $e$ and $\{c',d'\}$ for $f$.
        \item Two inessential quasicircles $e$ and $f$ are disjoint and nested if and only if, up to relabeling, for any bigon pair $\{c, d\}$ for $e$, there exists a bigon pair $\{c',d'\}$ for $f$ disjoint from $c$.
        \item An essential quasicircle $a$ and inessential quasicircle $e$ are disjoint if and only if there exists a bigon pair $\{c, d\}$ for $e$ that is disjoint from $a$.
    \end{enumerate}
\end{lemma}

\begin{proof}
Let $e, f \in \eqcg(S)$ be two disjoint inessential quasicircles. If $e,f$ are not nested, then $e$ is contained in the $S^1_g$ component of $S\setminus f$. So there exists a bigon pair $\{c,d\}$ for $e$ in this component that will be disjoint from $f$. Moreover, since bigon pairs consist of nonseparating quasicircles and are closed subsets of $S$, then there exists a bigon pair $\{c',d'\}$ for $f$ on $S\setminus \{c\cup d\}$. The reverse direction is straightforward. The existence of disjoint bigon pairs prevents any nesting since there do not exist any bigon pairs in the disk component of $S \setminus e$ or $S \setminus f$. 

Now suppose that $e,f$ are nested. By possibly relabeling, we can now suppose that $f$ is nested in the disk component of $S \setminus e$. Let $\{c, d\}$ be any bigon pair for $e$. Since $c$ is a nonseparating curve and $S \setminus c$ is an open subset of $S$ containing $f$, then there exists a bigon pair $\{c',d'\}$ for $f$ that is disjoint from $c$. For the reverse direction, there are two cases to consider. When $e$ and $f$ intersect or are disjoint, but not nested, then there exists a bigon pair $\{c,d\}$ for $e$ where both $c$ and $d$ intersect $f$, either along $e$ or outside of $e$. In both cases, it is impossible to find a bigon pair for $f$ that is disjoint from $c$.

Finally, we come to the case of an essential quasicircle $a$ and inessential quasicircle $e$. Suppose they are disjoint. Since $a$ is essential, then there is enough topology in the component of $S\setminus a$ that contains $e$ to contain some bigon pair $\{c,d\}$ for $e$. The reverse direction follows from the fact that $e$ is a subset of its bigon pair.
\end{proof}

\begin{proof}[Proof of Proposition~\ref{Prop:AutQuasiGraphToAutExtended}]
    Let $\Phi \in \Aut \qcg(S)$. We will define the map $\overline{\Phi} \in \Aut\eqcg(S)$ by the following two rules:
    \begin{enumerate}
        \item If $a \in \eqcg(S)$ is an essential quasicircle, then $\overline{\Phi}(a)=\Phi(a)$.
        \item If $e \in \eqcg(S)$ is an inessential quasicircle, then let $\{c,d\}$ be a bigon pair for $e$. Define $\overline{\Phi}(e)$ to be the inessential quasicircle determined by $\{\Phi(c),\Phi(d)\}$.
    \end{enumerate}
Directly from this definition, it is clear that the action of $\overline{\Phi}$ restricted to the essential vertices is exactly $\Phi$.

The inessential curve determined by $\{\Phi(c),\Phi(d)\}$ does not depend on the choice of bigon pair as Corollary~\ref{Cor:LinkedQuasiArcGraphConnected} implies that there exists a path of bigon pairs from $\{c,d\}$ to any other bigon pair. In addition, we have that each consecutive pair in this path forms a linked sharing pair, and linked sharing pairs are preserved by Proposition~\ref{Prop:SharingPairsPreserved}. Thus $\overline{\Phi}$ is well-defined.

Moreover, Lemma~\ref{Lem:EdgeBetweenExtendedGraph} shows that for any two vertices $a,b \in\eqcg(S)$, there is an edge between $a$ and $b$ if and only if there is an edge between $\overline{\Phi}(a)$ and $\overline{\Phi}(b)$.

Finally, $\overline{{\Phi}^{-1}}$ induces exactly the inverse action of $\overline{\Phi}$, namely
$$\overline{\Phi^{-1}}(\overline{\Phi}(a)) = a \qquad \mbox{ and } \qquad \overline{\Phi^{-1}}(\overline{\Phi}(e)) = e$$ since $\{\Phi(c),\Phi(d)\}$ is a bigon pair for $\overline{\Phi}(e)$  and $\{\Phi^{-1}(\Phi(c)),\Phi^{-1}(\Phi(d))\} = \{c,d\}$, which is a bigon pair for $e$. Thus $\overline{\Phi}$ is invertible and is a graph automorphism of $\eqcg(S)$ as desired.
\end{proof}

\subsection{Building the homeomorphism}\label{ss-get-homeo}
While the previous sections follow the work in \cite{long2023automorphisms}, their map from automorphisms of the extended graph to the homeomorphisms of the surface cannot be directly adapted to quasicircles.
The fundamental obstacle here is that through any convergent sequence of distinct points there is a Jordan curve by the Denjoy--Riesz Theorem \cite{mooreKline1919denjoyRiesz}, but not necessarily a more regular curve.
Specifically, \cite[Lemma 4.1]{long2023automorphisms} uses the classification of infinite-type surfaces \cite[Theorem 1]{richards1963classification} to prove this fact, which is only up to homeomorphism.

Fortunately, this obstacle has been overcome for smooth curves in ~\cite{booth2024automorphisms}. We will  outline this argument to show the following:

\begin{proposition}\label{prop-auts-come-from-homeos}
    Let $\Phi\in \Aut\mathcal{QC}^{\dagger}(S)$. Then $\Phi = \varphi_{*}$ for some $\varphi \in {\rm{Homeo}}(S)$, where $\varphi_*$ is the automorphism of $\mathcal{QC}^\dagger(S)$ induced by $\varphi$.
\end{proposition}

We begin with identifying a graph structure in the extended graph that is the same as the graph structure in the non-extended graph.

\begin{lemma}\label{Lem:SepCurveinExtendGraph}
    Let $S$ be an oriented surface.  Let $a \in \eqcg(S)$. Then $a$ is separating if and only if $link(a)$ is a join.
\end{lemma}

The proof of this lemma follows directly from the proof of Lemma~\ref{Lem:SeparatingLinkJoin}.  Our next steps involve defining and identifying a construction that allows us to determine the homeomorphism of the surface.

\begin{definition}[Nested sequences]
    A sequence of curves $(c_i)$ on a surface $S$ is {\rm{nested}} if each $c_i$ is separating and there exists a component $C_i$ of $S \setminus c_i$ such that $c_j \subset C_i$, for all $j > i$. 
\end{definition}

\begin{lemma}\label{Lem:NestedCurvesPreserved}
    Let $\Psi \in \Aut \eqcg(S)$. If $(c_i) \in \eqcg(S)$ is a nested sequence of curves, then the sequence $(\Psi(c_i))$ is also nested. 
\end{lemma}

The proof for Lemma~\ref{Lem:NestedCurvesPreserved} follows directly from Lemma~\ref{Lem:SepCurveinExtendGraph}. We omit this proof as the details for the same result for smooth curves can be found in~\cite[Lemma 2.3]{booth2024automorphisms}.

\begin{definition}[Convergent sequences]
    A sequence of curves $(c_i)$ in $S$ will be said to {\rm{converge}} to a set $P$ if they converge to $P$ in the Hausdorff topology.  We say that a curve {\rm{intersects the tail}} of a sequence $(c_i)$ if it intersects infinitely many of the $c_i$. 
\end{definition}

\begin{lemma}\label{Lem:NestedConvergSeqInGraph}
    A sequence of nested curves $(c_i) \in \eqcg(S)$ converges to a point $x \in S$ if and only if both of the following hold:
    \begin{enumerate}[noitemsep,topsep=0pt, label=$(\roman*)$]
        \item there exists a curve $a \in \eqcg(S)$ that intersects the tail of $(c_i)$
        \item any two curves $a, b \in \eqcg(S)$ that intersect the tail of $(c_i)$ must intersect each other.
    \end{enumerate}
\end{lemma}

Again, the proof of this lemma is identical to the smooth version in~\cite[Lemma 2.4]{booth2024automorphisms} and is omitted here (c.f. \cite{farbMargalit2022Unpublished} for the continuous case of this and Lemma \ref{Lem:ExtendedAutToHomeo}). Since both of the above conditions are identifiable in the graph, this shows that nested convergent sequences are preserved by automorphisms of $\eqcg(S)$.

\begin{lemma} \label{Lem:ExtendedAutToHomeo}
    Every $\Psi \in \Aut \eqcg(S)$ induces a homeomorphism $\varphi_{\Psi}$ such that the action of $\varphi_{\Psi}$ on $\eqcg(S)$ is given by $\Psi$.
\end{lemma}

We will only sketch this proof to give the reader an idea of how this map is built and why it is a homeomorphism.  Full details can be found in the proof for the smooth case in~\cite[Lemma 2.5]{booth2024automorphisms}.

\begin{proof}[Proof Sketch]
    To define the desired map, we first take a nested convergent sequence $(c_i)$ that converges to a point $x \in S$.  Since $\Psi$ preserves nested convergent sequences, the image sequence $(\Psi(c_i))$ must also be a nested convergent sequence.  We define $\varphi_{\Psi}(x)$ to be the point of convergence for this image sequence. 
	
    This map is shown to be well-defined by interleaving subsequences of nested sequences that converge to the same point.  A similar interleaved subsequence of the image sequences is used to show that this map is injective.  Surjectivity follows in a straightforward manner by using a preimage of a nested convergent sequence and from the fact that $\Psi^{-1}$ is also an automorphism of the graph.
	
    To show continuity, we  construct a sequence of nested sequences that converge to the points in a sequence $\{x_n\}$ that converge to $x$.  A diagonal sequence is then used to show that the points $\{\varphi_{\Psi}(x_n)\}$  converge to $\varphi_{\Psi}(x)$.  Continuity of the inverse is immediate from the fact that $\varphi_{\Psi}^{-1} = \varphi_{\Psi^{-1}}$.
	
    Finally,  we  show that the action of $\varphi_{\Psi}$ on $\eqcg(S)$ is exactly $\Psi$. This is accomplished by defining nested convergent sequences for each point on a curve $c$.  Since $c$ intersects the tail of this convergent sequence, then $\Psi(c)$  also intersects the tail of the image sequence.  Thus by compactness, $\Psi(c)$  contains all the points in $\varphi_{\Psi}(c)$. 
\end{proof}

We now give the final proof of this subsection.

\begin{proof}[Proof of Proposition~\ref{prop-auts-come-from-homeos}]
    By Proposition~\ref{Prop:AutQuasiGraphToAutExtended},  every element $\Phi$ of $\Aut \qcg(S)$ corresponds to an element $\overline{\Phi}$ of $\Aut \eqcg(S)$ such that the restriction of $\overline{\Phi}$ to $\qcg(S)$ is exactly $\Phi$.  By Lemma~\ref{Lem:ExtendedAutToHomeo},  there is a homeomorphism $\varphi_{\overline{\Phi}}$ such that the action of $\varphi_{\overline{\Phi}}$ on the extended graph is exactly $\overline{\Phi}$.  Thus the action of $\varphi_{\overline{\Phi}}$ on $\qcg(S)$ is exactly $\Phi$, as desired.
\end{proof}

\subsection{Quasiconformality}\label{ss-quasiconformality}
We now turn to examining the quasiconformality of the homeomorphisms inducing automorphisms of $\eqcg(S)$.
The structure of our argument is to apply Aseev's theorem after constraining the quasicircle constants of images of sufficiently small circles in $S$ in terms of the $4$-point bounded turning condition.

The following is the observation that we use to go from information about essential curves to small circles.

\begin{lemma}\label{lemma-stick-arc-in-QC}
    Let $(S,g)$ be a hyperbolic surface.
    Then there is a finite family $F$ of isotopy classes of simple closed curves and $K, \varepsilon > 0$ so that for all circles $C$ on $S$ of diameter less than $\varepsilon$ and circularly ordered quadruples $q = (x_1, x_2, x_3,x_4) $ in $C$ there is an arc $\mathcal{A}$ of $C$ containing $q$ and an essential $K$-quasicircle $\gamma$ with isotopy class in $F$ containing $\mathcal{A}$.
\end{lemma}

Working only with metric circles here will suffice for our uses because Aseev's Theorem \ref{thm-aseev} only requires verification that circles are mapped to quasicircles.

\begin{proof}
    Let $\varepsilon$ be small compared to the injectivity radius of $(S,g)$
    and let $C$ be a circle of diameter less than $\varepsilon$. Then there is a $(3/4)$-circle arc $\mathcal{A}_q$ of $C$ containing $q$: this is the key point.
    
    Now extend $\mathcal{A}_q$ to an arc $\alpha_q$ that has diameter a definite but small fraction of the injectivity radius by extending each side with geodesic segments perpendicular to the boundary of the circle.
    Let $B_q$ be a ball containing $\alpha_{q}$ and so that the endpoints of $\alpha_{q}$ are contained in and perpendicular to $\partial B_{q}$. Finally, obtain $\gamma$ by connecting the sides to each other with an arc that only meets $B_{q}$ at the endpoints and lies in one of a finite family $F$ of isotopy classes.

    We conclude by noting that this process yields $\gamma$ with uniformly controlled quascircularity constant $K$.
    Because the curves $\gamma$ are so regular and controlled this is quite standard; we outline a careful proof for the convenience of the reader.

    Let $\nu$ be the geodesic in the isotopy class of $\gamma$.
    Our task is to find a quasiconformal map $\varphi : (S,g) \to (S,g)$ taking $\nu$ to $\gamma$ with uniform control on $K(\varphi)$. 
    The above process allows us to choose $\gamma$ with the following control.
    First, we arrange for $\gamma$ to be piecewise smooth and piecewise geodesic outside $\mathcal{A}_q$, with uniform bounds on the number of geodesic segments and angles between segments.
    Next, we assure that $\gamma$ does not wrap too closely back on itself in the sense that for a uniform $\varepsilon > 0$, the $\varepsilon$-neighborhood of the essential lift of $\gamma$ in the corresponding annular covering is embedded isometrically in $(S,g)$ by the covering map.

    To extend to the rest of $S$, we define a map on metric graphs in $S$ whose complementary regions are disks (e.g. topological triangulations), then apply extension results to define this map on the remaining disks.
    Because of the uniformity of our construction, we may take two embedded graphs $G_1, G_2$ with peicewise smooth edges in $S$ satisfying the following:
    \begin{enumerate}
        \item $\nu$ and $\gamma$ are unions of cycles $C_i$ in $G_i$ of the same length. 
        \item All complementary components $P_j^i$ of $G_i$ are disks ($i = 1, 2$, $j = 1, ..., n$) with uniformly bounded diameter in their induced metrics,
        \item There is a homeomorphism $\psi$ of $S$ isotopic to the identity taking $G_1$ to $G_2$.
        Write the induced isomorphism of graphs by $\Phi$.
        \item There is a homeomorphism $\psi' : G_1 \to G_2$ inducing the graph isomorphism $\Phi$ so that the restriction of $\psi'$ to the boundary of each cell $P_j^i$ is $L$-bilipschitz for a uniform $L > 1$ with respect to the metrics on $P_j^i$ induced by the metric on $S$.
    \end{enumerate} 

    It then follows from lifting to $\bbH^2$, comparing the hyperbolic and Euclidean metrics with the diameter bound on $P_j^i$, and applying Theorem \ref{thm-bilip-extension} that there is a uniform $K < \infty$ so that there exist $K$-quasiconformal homeomorphisms $P_j^1 \to P_j^2$ that restrict to $\psi'$ on $\partial P_j^i$.
    We conclude that $\psi'$ extends to a homeomorphism $\psi''$ of $S$ that takes $\nu$ to $\gamma$ and is $K$-quasiconformal on the complement of $G$.
    The map $\psi''$ is then $K$-quasiconformal as the graphs $G_i$ have measure $0$.
\end{proof}

We now prove Proposition \ref{prop-is-quasiconf}.

\begin{proof}
    Let $S = S_g$ with $g \geq 2$ and let $\Phi\in\mathrm{CAut} \mathcal{QC}^{\dagger}(S)$.
    By Proposition \ref{prop-auts-come-from-homeos}, $\Phi$ is induced by a homeomorphism $\varphi$ of $S$.
    
    Since $K$-quasiconformality is a local condition, it suffices to produce for every point $p \in S$ a neighborhood $U$ of $p$ so that the restriction of $\varphi$ to $U$ is $K$-quasiconformal.

    Let $\varepsilon_0>0$ be sufficiently small that Lemma \ref{lemma-stick-arc-in-QC} holds for some $K_0$ and finite set $F \subset \mathscr{S}$ of isotopy classes of simple closed curves.
    As $\Phi$ is weakly coarse order preserving and $F$ is finite, there is a $K_1$ so that $\varphi(\gamma)$ is a $K_1$-quasicircle for every $K_0$-quasicircle $\gamma$.

    Uniformize $(S,g)$ so that the universal cover of $S$ is identified with the hyperbolic plane $\bbH^2$ in the Poincar\'e disk model.
    We may lift $\varphi$ to an equivariant homeomorphism $\widetilde{\varphi}$ of $\bbH^2$.
    To verify quasiconformality of $f$ near $p$, it suffices to verify quasiconformality of $\widetilde{\varphi}$ near a lift of $p$.

    Now let $C$ be a circle inside of $U$, let $q$ be an ordered quadruple of points in $C$ and let $\alpha$ be the curve given by Lemma \ref{lemma-stick-arc-in-QC}. 
    By the definition of $K$-quasicircles, there is a $K_1$-quasiconformal homeomorphism $\eta : S \to S$ and a geodesic $\gamma \in F$ so that $\eta(\gamma) = \varphi(\alpha)$.
    Note that $\eta$ lifts to a $K_1$-quasiconformal homeomorphism of $\bbH^2$, and hence extends to a $K_1^2$-quasiconformal homeomorphism of $\CP^1$, and that the extension of $\eta$ maps a circle in $\CP^1$ to a $K_1^2$-quasicircle containing $\widetilde{\alpha}$.
    We conclude that for any lift $\widetilde{\alpha}$ of $\alpha$ to $\bbH^2$ that $\widetilde{\varphi}(\widetilde{\alpha})$ satisfies the $c(K_1^2)$-$4$-point bounded turning condition, and in particular the $c(K_1^2)$-$4$-point bounded turning condition holds for lifts of our quadruple of points in $C$.
    
    Since this holds for an arbitrary quadruple of points in $C$, we conclude that $\widetilde{\varphi}(C)$ is a $K(c(K_1^2))$-quasicircle in $\CP^1$.
    By Aseev's theorem, $\widetilde{\varphi}$ is quasiconformal with a $p$-independent quality constant in a neighborhood of $p$, as desired.
\end{proof}

\section{Hyperbolicity}

In this section, we prove that the quasicircle graph is hyperbolic (Theorem \ref{thm-hyperbolic}).
This is a fundamental property for graphs acted on by groups in geometric-group-theoretic applications.
The proof closely follows the proof of hyperbolicity of the fine curve graph by Bowden--Hensel--Webb \cite{bowden2022quasimorphisms}, however, minor changes are required when working with quasicircles.

We will define hyperbolicity via the four point condition: for points $x,y,w$ of a metric space $(X,d)$, the Gromov product is defined to be
\[
\langle x,y \rangle_w = \frac{1}{2} (d(w,x) + d(w,y) - d(x,y)).
\]
The space $X$ is said to be $\delta$-hyperbolic if for all $w,x,y,z \in X$,
\[
\langle x,z \rangle_w \geq \min\{ \langle x,y \rangle_w, \langle y,z \rangle_w \} - \delta.
\]
We give all graphs metrics by setting each edge to have length equal to 1.

Throughout this section, let $S$ denote a closed oriented surface of genus $g \geq 2$.
First, we show that the quasicircle graph is path connected.

\begin{proposition}\label{thm:pathconnected}
    $\qcg(S)$ is path connected.
\end{proposition}

\begin{proof}
    Notice that there is a simplicial map $\qcg(S) \to \calC(S)$ given by taking isotopy classes. Thus, it suffices to show that for any vertex of $\calC(S)$, the subgraph of $\qcg(S)$ spanned by its preimage is connected.
    This follows directly from the existence of a path between homotopic quasicircles in Lemma \ref{Lem:QuasiCirclesHomotopic}. We note that the quality constants of the vertices in such a path  may be larger than the quality constants of its endpoints.
\end{proof}

We now introduce the \textit{surviving curve graph}, which is a subgraph of the curve graph.
The distance in the surviving curve graph models distance in the quasicircle graph, which helps us prove the quasicircle graph is hyperbolic.

\begin{definition}[Surviving curve graph]
Let $n \geq 0$ and $\Sigma$ be the surface of genus $g \geq 1$ with $n$ punctures. The {\rm{surviving curve graph}} of the surface $\Sigma$, denoted $\mathcal{C}^s(\Sigma)$, is the
subgraph of $\calC(\Sigma)$ spanned by all vertices corresponding to curves that continue to be essential
in the closed surface obtained by filling in the punctures of $\Sigma$.
\end{definition}

Work of Rasmussen shows that the surviving curve graph is $\delta$-hyperbolic \cite{rasmussen}:

\begin{theorem}\label{Thm:SurvivingHyperbolic}
    Let $\Sigma$ be a finite-type surface with positive genus. Then
    there is a number $\delta > 0$ such that $\mathcal{C}^s(\Sigma)$ is $\delta$-hyperbolic.
\end{theorem}

The following two lemmas, found in Bowden--Hensel--Webb \cite[Lemmas 3.5 \& 3.6]{bowden2022quasimorphisms} in the context of all curves, allow us to prove that  distance in the surviving curve graph and  distance in the quasicircle graph are equivalent.
Both proofs only require a minor modification to account for the quality constants of the quasicircles.
In particular, when applying isotopies and perturbations to curves in the proofs, we note that the quality constants of the quasicircles may increase.

\begin{lemma}\label{Lem:IsotopeToMinimalPosition}
    Let $P \subset S$ be a finite set.
    Suppose that $\alpha_1, \ldots, \alpha_n$ are curves that are pairwise in minimal position in $S - P$. Let $\beta_1, \ldots , \beta_m$ be $K_i$-quasicircles that are disjoint from $P$.
    Then the $\beta_i$ can be isotoped to $K_i'$-quasicircles $\beta_1', \ldots , \beta_m'$ in $S - P$ such that $\alpha_1, \ldots , \alpha_n, \beta_1', \ldots , \beta_m'$ are
    pairwise in minimal position in $S - P$.
\end{lemma}

\begin{lemma}\label{Lem:GeodesicDisjointFromP}
    Let $\alpha, \beta \in \qcg(S)$ and $P \subset S$ be a finite set. Then we may find
a geodesic $\alpha = v_0, \ldots , v_k = \beta$ such that $v_i \cap P = \emptyset$ for all $0 < i < k$.
\end{lemma}

Notice that Lemma \ref{Lem:IsotopeToMinimalPosition} shows that the distance in the surviving curve graph is an upper bound for the distance in the quasicircle graph, and Lemma \ref{Lem:GeodesicDisjointFromP} shows that the distance in the surviving curve graph is a lower bound for the distance in the quasicircle graph, giving us the following: 

\begin{lemma}\label{Lem:DistanceinQCSurviving}
    Suppose that $\alpha, \beta \in \qcg(S)$ intersect transversely, and that $\alpha$ and $\beta$
    are in minimal position in $S -P$ where $P \subset S$ is finite and disjoint from $\alpha$
    and $\beta$. Then
    $d_{\mathcal{C}^s(S-P)}
    ([\alpha]_{S-P} , [\beta]_{S-P} ) = d_{\qcg(S)}
    (\alpha, \beta)$.
\end{lemma}

Using that the distance between two quasicircles in the surviving curve graph is equivalent to their distance in the quasicircle graph, we are now in a position to prove that the quasicircle graph is hyperbolic.

\begin{proof}[Proof of Theorem~\ref{thm-hyperbolic}]
    Let $\alpha, \beta, \gamma, \mu \in \qcg(S)$. We will show that
    \[
        \langle \alpha, \gamma\rangle_{\mu} \geq \min\{ \langle \alpha, \beta\rangle_{\mu}, \langle \beta, \gamma\rangle_{\mu} \} - \delta - 4,
    \]
    where $\delta$ is the hyperbolicity constant for $\mathcal{C}^{s}(S -P)$. To do this we construct vertices $\alpha', \beta', \gamma', \delta' \in \mathcal{C}^{s}(S -P)$ which are all (i) transverse, and such that (ii) $d_{\qcg(S)}(\alpha,\alpha') \leq 1$, $d_{\qcg(S)}(\beta,\beta') \leq 1$, $d_{\qcg(S)}(\gamma,\gamma')\leq 1$.
    Set $\mu = \mu'$.
    To ensure (i) holds, we first find $\alpha''$ disjoint from and isotopic to $\alpha$, and we then apply a sufficiently small perturbation to $\alpha''$ so that (ii) holds.
    We note that the quality constant of $\alpha'$ may be larger than the quality constant $\alpha$.
    We repeat the process for $\beta$ and $\gamma$.
    We then choose a finite subset of points $P \subset S$ so that any bigon between $\alpha', \beta', \gamma',$ and $\mu'$ contains a point of $P$, which ensures each of $\alpha', \beta', \gamma',$ and $\mu'$ are in minimal position.
    Applying Lemma \ref{Lem:DistanceinQCSurviving} and Theorem \ref{Thm:SurvivingHyperbolic} proves our claim.
\end{proof}

\bibliographystyle{plain}
\bibliography{refs}

\end{document}